\documentclass[a4paper,11pt,DIV12]{scrartcl}

\usepackage{graphicx,subfigure,booktabs,wrapfig,amsmath,amssymb,placeins,enumerate}
\usepackage[pdftex,usenames]{color}
\usepackage[pdftex,bookmarks=false,pdfpagelabels=true,plainpages=false]{hyperref}
\usepackage{enumitem}
\usepackage{wrapfig,url}

\usepackage{multirow}%
\usepackage{amsfonts}%
\usepackage{amsthm}%
\usepackage{mathrsfs}%
\usepackage[title]{appendix}%
\usepackage{xcolor}%
\usepackage{textcomp}%
\usepackage{manyfoot}%
\usepackage{booktabs}%
\usepackage{algorithm}%
\usepackage{algorithmicx}%
\usepackage{algpseudocode}%
\usepackage{listings}%

\usepackage{todonotes}
\usepackage{tikz}

\newcommand{\abs}[1]{\lvert #1 \rvert}
\newcommand{\bigO}{\mathcal{O}}
\newcommand{\dx}{h}
\newcommand{\dy}{\Delta y}
\newcommand{\dt}{\Delta t}
\newcommand{\gnn}{\gamma^n_0}
\newcommand{\gni}{\gamma^n_i}
\newcommand{\TestCart}{3}
\newcommand{\Testcut}{4}

\newtheorem{assumption}{\textbf{Assumption}}[section]

\newtheorem{Error_sources}{Error sources}[section]

\definecolor{OliveGreen}{rgb}{0,0.39,0}
\definecolor{azure}{rgb}{0.0, 0.5, 1.0}

\newtheorem{theorem}{Theorem}
\newtheorem{lemma}{Lemma}
\newtheorem{remark}{Remark}%

\title{Accuracy analysis for explicit-implicit finite volume schemes on cut cell meshes}

\author{Sandra May\thanks{Department of Inf. Tech., Uppsala University, Box 337, 751 05 Uppsala, Sweden, sandra.may@it.uu.se} \and Fabian Laakmann\thanks{formerly: Mathematical Institute, University of Oxford, OX2 6GG Oxford, United Kingdom, fabian.laakmann@maths.ox.ac.uk}}

\date{}

\begin{document}

\maketitle

\begin{abstract}The solution of time-dependent hyperbolic conservation laws on cut cell meshes causes the small cell problem: standard schemes are not stable on the arbitrarily small cut cells if an explicit time stepping scheme is used and the time step size is chosen based on the size of the background cells. In [J. Sci. Comput. 71, 919–943 (2017)],
  the mixed explicit implicit approach in general and MUSCL-Trap in particular have been introduced to solve this problem by using implicit time stepping on the cut cells. Theoretical and numerical results have indicated that this might lead to a loss in accuracy when switching between the explicit and implicit time stepping. In this contribution we examine this in more detail and will prove in one dimension that the specific combination MUSCL-Trap of an explicit second-order and an implicit second-order scheme results in a fully second-order mixed scheme. As this result is unlikely to hold in two dimensions, we also introduce two new versions of mixed explicit implicit schemes based on exchanging the explicit scheme. We present numerical tests in two dimensions where we compare the new versions with the original MUSCL-Trap scheme.
  \end{abstract}

\section{Introduction}\label{sec: Intro}

Cartesian embedded boundary meshes for computing flow problems involving complex geometries  have become very popular as mesh generation is fully automatic and fairly cheap: the geometry is simply cut out of Cartesian background cells. This results in \textit{cut cells} where the object intersects the background mesh. Cut cells can have various shapes and can in particular be arbitrarily small. This causes various issues when standard methods are used to solve partial differential equations (PDEs) on these meshes. The specific issue depends on the kind of PDE to be solved.

In the context of time-dependent hyperbolic conservation laws the main issue is what is referred to as the \textit{small cell problem}: Typically, explicit time stepping schemes are used. By the CFL condition the time step size is coupled to the size of the cells. One wants to choose the time step based on the size of the larger background cells and use the same time step on the arbitrarily small cut cells. For standard schemes, this approach causes the values on the small cut cells to explode.

One approach to solve this issue is to use \textit{cell merging} or \textit{cell agglomeration}, see, e.g., \cite{Krivodonova2013,Kummer2016,Quirk1994}.
In that approach cut cells that are too small are merged / combined with neighboring cells, which eliminates the problem. The downside is that the complexity is moved back into mesh generation process. The alternative is to develop \textit{algorithmic} solution approaches to overcome the small cell problem. Two very well-established approaches are the \textit{flux redistribution method} \cite{Chern_Colella,Colella2006} and the $h$-box method \cite{Berger_Helzel_Leveque_2003,Berger_Helzel_Leveque_2005,Berger_Helzel_2012}. More recent approaches include the \textit{dimensionally split approach} \cite{Klein_cutcell,Klein_cutcell_3d}, the extension of the active flux method to cut cells \cite{FVCA_Helzel_Kerkmann}, and the \textit{state redistribution scheme} (SRD) \cite{Berger_Giuliani_2021}. All of these schemes are based on finite volume approaches (if one counts the active flux method as a finite volume scheme).
A few years ago, the development of algorithmic solution approaches in the context of discontinuous Galerkin (DG) methods has started. Existing solution approaches here are the \textit{DoD stabilization} \cite{DoD_SIAM_2020,DoD_AMC},
the extension of the \textit{ghost penalty stabilization} \cite{Burman2010} to time-dependent first-order hyperbolic problems
\cite{Kreiss_Fu,2d_ghostpenalty},
as well as the extension of the SRD scheme to a DG setting \cite{Giuliani_DG}.

In this contribution we will consider the \textit{mixed explicit implicit} scheme in more detail, which was introduced by May and Berger \cite{May_PhD, SM_May_Berger_FVCA, May_Berger_explimpl} in the context of a finite volume setting to overcome the small cell problem. The idea is quite simple: cut cells are treated implicitly for stability but cells away from the cut cells use a standard explicit time stepping scheme to keep the cost low. The switch happens by means of \textit{flux bounding}. The authors combined a
second-order explicit scheme with a second-order implicit scheme. Numerical experiments 
\cite{May_Berger_explimpl} have shown that the resulting mixed scheme converges with second order in the $L^1$ norm. In the
$L^{\infty}$ norm, one numerically sees full second order in one dimension but in two and three dimensions orders between 1 and 2 were observed.

In this contribution we want to examine the error behavior of the mixed explicit implicit scheme in more detail and also test possible cures. In one dimension, we will first analyze the one step error and then show that the mixed scheme indeed converges with second order. As numerical results indicate that we will not be able to prove the same result in higher dimensions, we will then introduce two new versions of the mixed explicit implicit scheme with an improved transition error. These will be designed so that on a Cartesian mesh the transition between explicit and implicit time stepping has a third-order one step error. We will include results from a test involving cut cells to see whether this increases the overall accuracy of the scheme. Some of the material presented here has been part of the master thesis of Fabian Laakmann \cite{Laakmann2018}.

The idea of combining explicit and implicit time stepping to gain stability has been used by other authors as well and has in particular become more popular in recent years, see, e.g., 
\cite{Col_Col_Glaz,Mikula_Ohlberger_Urban, MUSCAT2019108883,Frolkovic_2022} and the references cited therein. Here, we will focus on the approach as introduced in \cite{May_PhD, May_Berger_explimpl} in the context of cut cells. We note though that this approach can easily be extended to other problem settings as well. Recently, it has been applied in the context of two phase flow problems: if one uses a sharp interface method, then the creation of new phases (\textit{nucleation} or \textit{cavitation}) results in the creation of new tiny cells. In \cite{May_Thein}, the newly created tiny cells have successfully been treated with the first-order version of the mixed explicit implicit scheme, extended to the isothermal Euler equations.

This contribution is structured as follows: in section \ref{sec: schemes 1d}, we focus on the situation in one dimension. We introduce the mixed explicit implicit scheme MUSCL-Trap as used in \cite{May_Berger_explimpl} and examine its error. We then present two different variants of the mixed explicit implicit scheme, which have a reduced transition error compared to the original scheme. In section \ref{sec: schemes 2d}, we extend the considerations to two dimensions. We first formulate the new variants in two dimensions and then compare the three schemes numerically.
We conclude with an outlook in section \ref{sec: outlook}.

\section{Mixed explicit implicit schemes in 1d}\label{sec: schemes 1d}

We consider the linear advection equation
\begin{equation}\label{eq: lin adv}
s(t,x)_t + u s(t,x)_x = 0, \quad u > 0 \:\: \text{constant},
\end{equation}
with initial data $s(t_0,\cdot) = s_0$.
The standard model mesh for developing cut cell schemes in 1d is shown in
figure \ref{Fig: 1d model problem}: an equidistant mesh with mesh width $\dx$ contains one cell
of length $\alpha \dx$, labeled as cell $0$, in the middle.
Here, $\alpha \in (0,1]$ denotes the {\em volume fraction} -- the ratio of the small cell to full cell volume.
 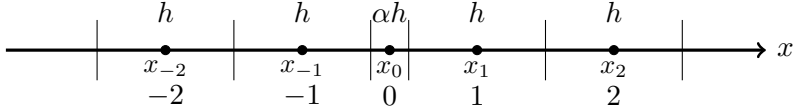
\begin{figure}
\begin{center}
\begin{tikzpicture}[
axis/.style={very thick, line join=miter, ->}]
\draw [axis] (-4.8,0) -- (5.2,0) node(xline)[right] {$x$};
\draw (-3.6,-0.4) -- (-3.6,0.4);
\draw (-1.8,-0.4) -- (-1.8,0.4);
\draw (0,-0.4) -- (0,0.4);
\draw (0.5,-0.4) -- (0.5,0.4);
\draw (2.3,-0.4) -- (2.3,0.4);
\draw (4.1,-0.4) -- (4.1,0.4);
\node[] at (-2.7,-0.6) {$-2$};
\node[] at (-0.9,-0.6) {$-1$};
\node[] at (0.25,-0.6) {$0$};
\node[] at (1.4,-0.6) {$1$};
\node[] at (3.2,-0.6) {$2$};
\node[] at (-2.7,0.5) {$h$};
\node[] at (-0.9,0.5) {$h$};
\node[] at (0.25,0.5) {{$\alpha h$}};
\node[] at (1.4,0.5) {$h$};
\node[] at (3.2,0.5) {$h$};
\draw[fill] (-2.7,0) circle (.06cm);
\draw[fill] (-0.9,0) circle (.06cm);
\draw[fill] (0.25,0) circle (.06cm);
\draw[fill] (1.4,0) circle (.06cm);
\draw[fill] (3.2,0) circle (.06cm);
\node[] at (-2.7,-0.25) {\small $x_{-2}$};
\node[] at (-0.9,-0.25) {\small $x_{-1}$};
\node[] at (0.25,-0.25) {\small $x_{0}$};
\node[] at (1.4,-0.25) {\small $x_{1}$};
\node[] at (3.2,-0.25) {\small $x_{2}$};
\end{tikzpicture}
\end{center}
\caption{1d model problem: Equidistant grid with one small cell of length $\alpha h$ 
labeled as cell 0.}
\label{Fig: 1d model problem}
\end{figure}

\subsection{The MUSCL-Trap scheme}

The idea behind mixed explicit implicit schemes for cut cell meshes as used in \cite{May_PhD, SM_May_Berger_FVCA, May_Berger_explimpl} is to employ an implicit scheme on the cut cell $0$ for stability. Away from the cut cell
an explicit scheme is used to keep the cost low.
In \cite{May_Berger_explimpl}, May and Berger developed what we will refer to as {\em MUSCL-Trap}, the combination of the explicit MUSCL scheme and the implicit Trapezoidal scheme by means of
{\em flux bounding}.

MUSCL (Monotonic Upwind Scheme for Conservation Laws)
\cite{van_Leer_V,colella1985} is a common explicit finite volume scheme.  The scheme is second-order accurate in space and time.  
In one dimension, the MUSCL scheme for the linear advection equation \eqref{eq: lin adv} 
on an equidistant grid is given by
\begin{equation}\label{eq: MUSCL 1d}
S_i^{n+1} = S_i^n - \frac{\dt}{\dx} \left( F_{i+1/2}^{n+1/2,M} - F_{i-1/2}^{n+1/2,M} \right),
\quad
F_{i+1/2}^{n+1/2,M} = u \left( S_i^n + (1 - \lambda) S_{i,x}^n \frac{\dx}{2} \right),
\end{equation}
with $S_{i,x}^n \approx \partial_x s (t^n, x_i)$, computed using 
central difference gradients and
a standard slope limiter
\cite{Leveque02}, with CFL number $\lambda = \tfrac{u \dt}{\dx}$. The scheme is stable for $0 < \lambda \le 1$ and would therefore
become unstable on the cell $0$ for $\alpha < 1$.

The implicit Trapezoidal rule in time in combination with a suitable slope reconstruction in space is given by
\begin{subequations}\label{eq: Trap 1d}
\begin{align}
  S_i^{n+1} &= S_i^n - \frac{\dt}{\dx_i} \left( F_{i+1/2}^{n+1/2,T} - F_{i-1/2}^{n+1/2,T} \right),\\
F_{i+1/2}^{n+1/2,T} &= \frac{u}{2} \left( S_i^n + S_{i,x}^n \frac{\dx_i}{2} + S_i^{n+1} + S_{i,x}^{n+1} \frac{\dx_i}{2}\right), 
\end{align}
\end{subequations}
for a non-equidistant mesh with cells of length $\dx_i$. 

\begin{figure}
\centering
\subfigure[Step 1: Update of explicit cells.]
{
\begin{tikzpicture}[scale=0.8,
axis/.style={very thick, line join=miter, ->}]
\draw [axis] (-6.3,2.2) -- (-6.3,3.9) node(xline)[right] {};
\node[] at (-6.8,2.5) {{\small $t^n$}};
\node[] at (-6.8,3.74) {{\small $t^{n+1}$}};
\draw[dotted] (-6.5,2.5) -- (-3.6,2.5);
\draw (-5.4,2.5) -- (5.9,2.5);
\draw[dotted] (4.1,2.5) -- (6.9,2.5);
\draw (-5.4,2.4) -- (-5.4,2.6);
\draw (-3.6,2.4) -- (-3.6,2.6);
\draw (-1.8,2.4) -- (-1.8,2.6);
\draw (0,2.4) -- (0,2.6);
\draw (0.5,2.4) -- (0.5,2.6);
\draw (2.3,2.4) -- (2.3,2.6);
\draw (4.1,2.4) -- (4.1,2.6);
\draw (5.9,2.4) -- (5.9,2.6);
\node[] at (-4.5,2.1) {{\footnotesize $-3$}};
\node[] at (-2.7,2.1) {{\footnotesize $-2$}};
\node[] at (-0.9,2.1) {{\footnotesize $-1$}};
\node[] at (0.25,2.1) {{\footnotesize $0$}};
\node[] at (1.4,2.1) {{\footnotesize $1$}};
\node[] at (3.2,2.1) {{\footnotesize $2$}};
\node[] at (5.0,2.1) {{\footnotesize $3$}};
\draw[fill] (-4.5,2.5) circle (.08cm);
\draw[fill] (-2.7,2.5) circle (.08cm);
\draw[fill] (-0.9,2.5) circle (.08cm);
\draw[fill] (0.25,2.5) circle (.08cm);
\draw[fill] (1.4,2.5) circle (.08cm);
\draw[fill] (3.2,2.5) circle (.08cm);
\draw[fill] (5.0,2.5) circle (.08cm);
\draw[dotted] (-6.3,3.54) -- (-3.6,3.54);
\draw (-5.4,3.54) -- (5.9,3.54);
\draw[dotted] (4.1,3.54) -- (6.9,3.54);
\draw (-5.4,3.44) -- (-5.4,3.64);
\draw (-3.6,3.44) -- (-3.6,3.64);
\draw (-1.8,3.44) -- (-1.8,3.64);
\draw (0,3.44) -- (0,3.64);
\draw (0.5,3.44) -- (0.5,3.64);
\draw (2.3,3.44) -- (2.3,3.64);
\draw (4.1,3.44) -- (4.1,3.64);
\draw (5.9,3.44) -- (5.9,3.64);
\draw[fill=white] (-4.5,3.54) circle (.12cm);
\draw[fill=white] (-2.7,3.54) circle (.12cm);
\draw[fill=white] (3.2,3.54) circle (.12cm);
\draw[fill=white] (5.0,3.54) circle (.12cm);
\draw[line width = 0.7mm,OliveGreen] (-5.4,2.5) -- (-5.4,3.54);
\draw[line width = 0.7mm,OliveGreen] (-3.6,2.5) -- (-3.6,3.54);
\draw[line width = 0.7mm,OliveGreen] (-1.8,2.5) -- (-1.8,3.54);
\draw[dotted] (0,2.5) -- (0,3.54);
\draw[dotted] (0.5,2.5) -- (0.5,3.54);
\draw[line width = 0.7mm,OliveGreen] (2.3,2.5) -- (2.3,3.54);
\draw[line width = 0.7mm,OliveGreen] (4.1,2.5) -- (4.1,3.54);
\draw[line width = 0.7mm,OliveGreen] (5.9,2.5) -- (5.9,3.54);
\draw[line width = 0.7mm,OliveGreen] (-1.8,2.5) -- (-1.8,3.54);
\end{tikzpicture}
  \label{Fig: switch scheme: done with expl scheme}
}
\subfigure[Step 2: Update of cut cell and transition cells]
{
  \begin{tikzpicture}[scale=0.8,
axis/.style={very thick, line join=miter, ->}]
\draw [axis] (-6.3,-0.3) -- (-6.3,1.4) node(xline)[right] {};
%
%
\node[] at (-6.8,0) {{\small $t^n$}};
\node[] at (-6.8,1.24) {{\small $t^{n+1}$}};
\draw[dotted] (-6.5,0) -- (-3.6,0);
\draw (-5.4,0) -- (5.9,0);
\draw[dotted] (4.1,0) -- (6.9,0);
\draw (-5.4,-0.1) -- (-5.4,0.1);
\draw (-3.6,-0.1) -- (-3.6,0.1);
\draw (-1.8,-0.1) -- (-1.8,0.1);
\draw (0,-0.1) -- (0,0.1);
\draw (0.5,-0.1) -- (0.5,0.1);
\draw (2.3,-0.1) -- (2.3,0.1);
\draw (4.1,-0.1) -- (4.1,0.1);
\draw (5.9,-0.1) -- (5.9,0.1);
\node[] at (-4.5,-0.4) {{\footnotesize $-3$}};
\node[] at (-2.7,-0.4) {{\footnotesize $-2$}};
\node[] at (-0.9,-0.4) {{\footnotesize $-1$}};
\node[] at (0.25,-0.4) {{\footnotesize $0$}};
\node[] at (1.4,-0.4) {{\footnotesize $1$}};
\node[] at (3.2,-0.4) {{\footnotesize $2$}};
\node[] at (5.0,-0.4) {{\footnotesize $3$}};
\draw[fill] (-4.5,0) circle (.08cm);
\draw[fill] (-2.7,0) circle (.08cm);
\draw[fill] (-0.9,0) circle (.08cm);
\draw[fill] (0.25,0) circle (.08cm);
\draw[fill] (1.4,0) circle (.08cm);
\draw[fill] (3.2,0) circle (.08cm);
 \draw[fill] (5.0,0) circle (.08cm);
\draw[dotted] (-6.,1.04) -- (-3.6,1.04);
\draw (-5.4,1.04) -- (5.9,1.04);
\draw[dotted] (4.1,1.04) -- (6.9,1.04);
\draw (-5.4,0.94) -- (-5.4,1.14);
\draw (-3.6,0.94) -- (-3.6,1.14);
\draw (-1.8,0.94) -- (-1.8,1.14);
\draw (0,0.94) -- (0,1.14);
\draw (0.5,0.94) -- (0.5,1.14);
\draw (2.3,0.94) -- (2.3,1.14);
\draw (4.1,0.94) -- (4.1,1.14);
\draw (5.9,0.94) -- (5.9,1.14);
\draw[fill=white] (-4.5,1.04) circle (.12cm);
\draw[fill=white] (-2.7,1.04) circle (.12cm);
\draw[fill=white] (3.2,1.04) circle (.12cm);
\draw[fill=white] (5.0,1.04) circle (.12cm);
\draw[fill=white] (-1.02,0.92) rectangle (-0.78,1.16);
\draw[fill=white] (0.13,0.92) rectangle (0.37,1.16);
\draw[fill=white] (1.28,0.92) rectangle (1.52,1.16);
\draw[dotted] (-5.4,-0) -- (-5.4,1.04);
\draw[dotted] (-3.6,-0) -- (-3.6,1.04);
\draw[line width = 0.7mm,OliveGreen] (-1.8,-0) -- (-1.8,1.04);
\draw[line width = 0.7mm,azure] (0,-0) -- (0,1.04);
\draw[line width = 0.7mm,azure] (0.5,-0) -- (0.5,1.04);
\draw[line width = 0.7mm,OliveGreen] (2.3,-0) -- (2.3,1.04);
\draw[dotted] (4.1,-0) -- (4.1,1.04);
\draw[dotted] (5.9,-0) -- (5.9,1.04);
\draw[line width = 0.7mm,OliveGreen] (-1.8,-0) -- (-1.8,1.04);
\end{tikzpicture}
\label{fig: switch scheme: flux bounding}
}
\caption{Flux bounding: First the cells away from the cut cell are updated using an explicit scheme (edges marked in green). Then the neighborhood of the cut cell is
updated. The Cartesian neighbors of the cut cell are {\em transition cells}. The update on these cells employs both explicit fluxes (edges marked in green) and implicit fluxes (edges marked in light blue), see also \cite{May_Berger_explimpl},\cite[Fig.~6]{FVCA_May}. }
\label{fig: switch schemes}
\end{figure}
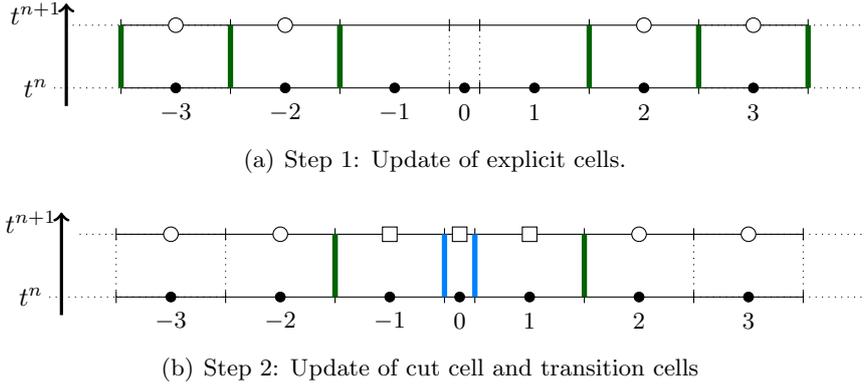
In \cite{May_PhD, May_Berger_explimpl}, May and Berger examined various ways for combining an explicit scheme with an implicit scheme and found {\em flux bounding} to be the best way. The idea is illustrated
in figure \ref{fig: switch schemes}: One first updates all Cartesian cells that are not direct neighbors of the cut cell using
the explicit scheme, see figure \ref{Fig: switch scheme: done with expl scheme}. In a second step, one updates the value
on cells $-1$, $0$, $1$, compare figure \ref{fig: switch scheme: flux bounding}. Here, cell $0$ is treated fully implicitly to guarantee stability on the small cut cell, indicated by the
light blue edges. The cells $-1$ and $1$ are {\em transition cells}:
When updating cell $-2$ in Step 1 with the fully explicit scheme, one assumed that a certain amount of mass, represented by $F_{-3/2}^{n+1/2,M}$,
has left cell $-2$. For the full scheme to be conservative, one needs to ensure that the mass arrives in cell $-1$. This is achieved
by re-using the explicit flux for the edge $-3/2$ in Step 2. Note that {\em transition cells} employ both explicit and implicit fluxes. Cut cells are treated
fully implicitly, and Cartesian cells that are {\em not} neighbors of cut cells are treated fully explicitly. Flux bounding is set up in a symmetric
manner, i.e., it does not distinguish between $u > 0$ and $u < 0$, allowing the extension to more general problems, such as Euler equations \cite{May_Thein}.
Flux bounding has the following properties:
\begin{enumerate}
\item It is conservative (by construction).
\item It satisfies a TVD property in a suitable setting. May and Berger \cite[Thm~1]{May_Berger_explimpl} showed TVD stability for a mixed scheme combining MUSCL (with minmod limiter \cite{Leveque02}) as explicit scheme with
  implicit Euler and piecewise constant data as implicit scheme for $0 < \lambda \le 1$,
  independent of the size of $\alpha \in (0,1]$.
\item It has a natural extension to 2d and 3d (see section \ref{sec: schemes 2d} as well as \cite{May_Berger_explimpl}).
  \end{enumerate}
This results in the following formulae for MUSCL-Trap in 1d:
\begin{equation}
\begin{aligned}\label{eq: mixed scheme in 1d}
\begin{split}
S_{-2}^{n+1} &= S_{-2}^n - \frac{\dt}{\dx}\left[ F_{-3/2}^{n+1/2,M}
- F_{-5/2}^{n+1/2,M} \right], & \text{(MUSCL)}, \\[0.3cm]
S_{-1}^{n+1} &= S_{-1}^n - \frac{\dt}{\dx} \left[ F_{-1/2}^{n+1/2,T}  - F_{-3/2}^{n+1/2,M}  \right],
& \text{(transition)},\\
S_0^{n+1} &= S_0^n - \frac{\dt}{\alpha \dx} \left[ F_{1/2}^{n+1/2,T} - F_{-1/2}^{n+1/2,T}  \right],
& \text{(Trapezoidal)},\\
S_1^{n+1} &= S_1^n - \frac{\dt}{\dx} \left[ F_{3/2}^{n+1/2,M} - F_{1/2}^{n+1/2,T}  \right],
& \text{(transition)},\\[0.3cm]
S_{2}^{n+1} &= S_{2}^n - \frac{\dt}{\dx} \left[ F_{5/2}^{n+1/2,M}
- F_{3/2}^{n+1/2,M}  \right], & \text{(MUSCL)}.
\end{split}
\end{aligned}
\end{equation}
For the slope reconstruction on cells $i \le -2$ and $i \ge 2$ unlimited central differences are used.
On the cut cell $0$ and the transition cells $-1$ and $1$ a least squares approach is employed \cite{Berger_Aftosmis_Murman}.
This extends in a straight-forward way to higher dimensions
\cite{Barth_LS_3d,May_Berger_LP}. 
Note that the least squares slope $S_{1,x}$ enters the computation of the MUSCL
flux $F_{3/2}^{n+1/2}$. Except for this special case, only central difference slopes are used for the computation
of MUSCL fluxes. Throughout this paper we will assume all slopes to be unlimited.

\begin{remark}
  This contribution focuses on the accuracy of mixed schemes in general and the accuracy of MUSCL-Trap in particular. We will not discuss limiting here. For limiting on cut cells, see, e.g., \cite{May_Berger_LP}.
  \end{remark}

\subsubsection{The one step error of MUSCL-Trap}
\label{subsec: one step error MUSCL Trap}

The one step error of scheme \eqref{eq: mixed scheme in 1d} has roughly been analyzed in \cite{May_Berger_explimpl, FVCA_May}. We will do this here more thoroughly.
Let the MUSCL-Trap scheme be given by $S_i^{n+1} = \Phi(S^n,S^{n+1})_i$ with
\begin{equation}\label{eq: def Phi}
\Phi(S^n,S^{n+1})_i = S_i^n - \frac{\dt}{\dx_i} \left[ F_{i+1/2}^{n+1/2,X}(S^n,S^{n+1}) - F_{i-1/2}^{n+1/2,X}(S^n,S^{n+1})  \right]
\end{equation}
with $F_{i+1/2}^{n+1/2,X}$ representing $F_{i+1/2}^{n+1/2,M}$ or $F_{i+1/2}^{n+1/2,T}$, depending on $i$.
For the computation of the one step error, we replace the input arguments of $\Phi_i$ 
with the true cell averages at time $t^n$ and $t^{n+1}$, given by $\bar{s}_i^n$ and $\bar{s}_i^{n+1}$, resulting in
\begin{equation}\label{eq: 1 step error}
  L(\bar{s}^{n},\bar{s}^{n+1})_i = \Phi(\bar{s}^n,\bar{s}^{n+1})_i - \bar{s}_i^{n+1}.
\end{equation}

\begin{remark}
  The {\em one step error} $L(\bar{s}^{n},\bar{s}^{n+1})_i$ measures the error made in one time step $t^n \to t^{n+1}$ on cell $i$.
Typically, i.e., on uniform meshes, the order of convergence of the one step error is one order higher than the overall order of the scheme.
\end{remark}

\begin{figure}
\centering
  \begin{tikzpicture}[scale=0.9,
axis/.style={very thick, line join=miter, ->}]
\draw [axis] (-6.3,-0.3) -- (-6.3,1.4) node(xline)[right] {};
%
%
\node[] at (-6.8,0) {{\small $t^n$}};
\node[] at (-6.8,1.24) {{\small $t^{n+1}$}};
\draw[dotted] (-6.5,0) -- (-3.6,0);
\draw (-5.4,0) -- (5.9,0);
\draw[dotted] (4.1,0) -- (6.9,0);
\draw (-5.4,-0.1) -- (-5.4,0.1);
\draw (-3.6,-0.1) -- (-3.6,0.1);
\draw (-1.8,-0.1) -- (-1.8,0.1);
\draw (0,-0.1) -- (0,0.1);
\draw (0.5,-0.1) -- (0.5,0.1);
\draw (2.3,-0.1) -- (2.3,0.1);
\draw (4.1,-0.1) -- (4.1,0.1);
\draw (5.9,-0.1) -- (5.9,0.1);
\node[] at (-4.5,-0.4) {{\footnotesize $-3$}};
\node[] at (-2.7,-0.4) {{\footnotesize $-2$}};
\node[] at (-0.9,-0.4) {{\footnotesize $-1$}};
\node[] at (0.25,-0.4) {{\footnotesize $0$}};
\node[] at (1.4,-0.4) {{\footnotesize $1$}};
\node[] at (3.2,-0.4) {{\footnotesize $2$}};
\node[] at (5.0,-0.4) {{\footnotesize $3$}};
\draw[fill] (-4.5,0) circle (.08cm);
\draw[fill] (-2.7,0) circle (.08cm);
\draw[fill] (-0.9,0) circle (.08cm);
\draw[fill] (0.25,0) circle (.08cm);
\draw[fill] (1.4,0) circle (.08cm);
\draw[fill] (3.2,0) circle (.08cm);
 \draw[fill] (5.0,0) circle (.08cm);
\draw[dotted] (-6.,1.04) -- (-3.6,1.04);
\draw (-5.4,1.04) -- (5.9,1.04);
\draw[dotted] (4.1,1.04) -- (6.9,1.04);
\draw (-5.4,0.94) -- (-5.4,1.14);
\draw (-3.6,0.94) -- (-3.6,1.14);
\draw (-1.8,0.94) -- (-1.8,1.14);
\draw (0,0.94) -- (0,1.14);
\draw (0.5,0.94) -- (0.5,1.14);
\draw (2.3,0.94) -- (2.3,1.14);
\draw (4.1,0.94) -- (4.1,1.14);
\draw (5.9,0.94) -- (5.9,1.14);
\draw[fill=white] (-4.5,1.04) circle (.12cm);
\draw[fill=white] (-2.7,1.04) circle (.12cm);
\draw[fill=white] (3.2,1.04) circle (.12cm);
\draw[fill=white] (5.0,1.04) circle (.12cm);
\draw[fill=white] (-1.02,0.92) rectangle (-0.78,1.16);
\draw[fill=white] (0.13,0.92) rectangle (0.37,1.16);
\draw[fill=white] (1.28,0.92) rectangle (1.52,1.16);
\draw[line width = 0.7mm,OliveGreen] (-5.4,-0) -- (-5.4,1.04);
\draw[line width = 0.7mm,OliveGreen] (-3.6,-0) -- (-3.6,1.04);
\draw[line width = 0.7mm,OliveGreen] (-1.8,-0) -- (-1.8,1.04);
\draw[line width = 0.7mm,azure] (0,-0) -- (0,1.04);
\draw[line width = 0.7mm,azure] (0.5,-0) -- (0.5,1.04);
\draw[line width = 0.7mm,OliveGreen] (2.3,-0) -- (2.3,1.04);
\draw[line width = 0.7mm,OliveGreen] (4.1,-0) -- (4.1,1.04);
\draw[line width = 0.7mm,OliveGreen] (5.9,-0) -- (5.9,1.04);
\node[] at (-5.4,-0.4) {{\footnotesize \textcolor{OliveGreen}{M}}};
\node[] at (-3.6,-0.4) {{\footnotesize \textcolor{OliveGreen}{M}}};
\node[] at (-1.8,-0.4) {{\footnotesize \textcolor{OliveGreen}{M}}};
\node[] at (-1.8,-0.4) {{\footnotesize \textcolor{OliveGreen}{M}}};
\node[] at (0,-0.4) {{\footnotesize \textcolor{azure}{T}}};
\node[] at (0.5,-0.4) {{\footnotesize \textcolor{azure}{T}}};
\node[] at (2.3,-0.4) {{\footnotesize \textcolor{OliveGreen}{M}}};
\node[] at (4.1,-0.4) {{\footnotesize \textcolor{OliveGreen}{M}}};
\node[] at (5.9,-0.4) {{\footnotesize \textcolor{OliveGreen}{M}}};
\node[] at (-4.5,0.5) {{\footnotesize $\mathcal{O}(h^3)$}};
\node[] at (-2.7,0.5) {{\footnotesize $\mathcal{O}(h^3)$}};
\node[] at (-0.9,0.5) {{\footnotesize $\mathcal{O}(h^2)$}};
\node[] at (0.25,0.5) {{\footnotesize $\mathcal{O}(h^2)$}};
\node[] at (1.4,0.5) {{\footnotesize $\mathcal{O}(h^2)$}};
\node[] at (3.2,0.5) {{\footnotesize $\mathcal{O}(h^2)$}};
\node[] at (5.0,0.5) {{\footnotesize $\mathcal{O}(h^3)$}};
\draw[line width = 0.7mm,OliveGreen] (-1.8,-0) -- (-1.8,1.04);
\end{tikzpicture}
\caption{The one step error $\abs{L(\bar{s}^{n},\bar{s}^{n+1})_i}$ for the MUSCL-Trap scheme for $i=-3,\ldots,3$. }
\label{fig: 1 step error}
\end{figure}
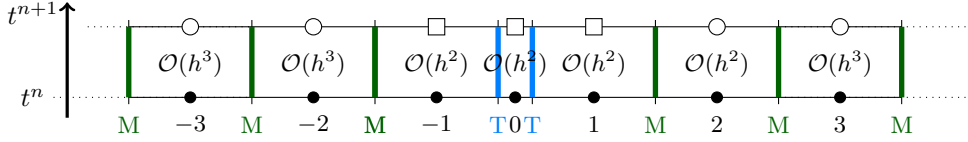

The numerically observed convergence order of the one step error for a smooth test problem is summarized in figure \ref{fig: 1 step error}.
We observe
a third-order one step error on cells that are sufficiently far away from the cut cell. These cells use the full MUSCL scheme on a uniform
mesh. 
In the neighborhood of the cut cell, we observe a second-order
one step error.
This is due to the following four error sources.
\begin{Error_sources}\label{error_sources}
\begin{itemize}
\item[{\textbf (1)}] The switch in the time stepping scheme between MUSCL and Trapezoidal rule results in an error term of size
  $\bigO(\dx^2)$. The examination of its propagation is one key aspect of this research.
\item[{\textbf (2)}] The irregularity of the cut cell causes several errors
  if $\alpha < 1$: 
  \begin{itemize}
  \item[(a)] the slope reconstruction, which uses a least squares approach, is only of first order, resulting in second-order error terms that we summarize in $a_i\dx^2$ for cell $i$; (this is the only reason for the error
    of size $\bigO(\dx^2)$ on cell $2$;)
  \item[(b)] in the computation of the fluxes, we use the variable $S_i^n$ (which is supposed to be an approximation to the cell average $\bar{s}_i^n$) as approximation to the point value
    $s(t^n,x_i)$; this results in second-order error terms that we summarize in $b_i\dx^2$;
  \item[(c)] for a third-order one step error, we would need to use a quadratic reconstruction in space instead of a linear one;
    we summarize the resulting error terms on cell $i$ in $c_i\dx^2$.
    \end{itemize}
  \end{itemize}
\end{Error_sources}
We note that on an equidistant mesh, i.e., $\alpha=1$, all error terms listed under (2) reduce to third-order errors, partly due
to cancellation effects with neighboring cells.
Using this notation, we observe on cells $-1$, $0$, and $1$ the following errors:
\begin{subequations}\label{eq: 1step central}
\begin{align}
L(\bar{s}^n,\bar{s}^{n+1})_{-1}&= \frac{1}{4}(1-\lambda)\lambda^2\dx^2 s_{xx}(t^n,x_{-1}) + a_{-1}\dx^2 + b_{-1}\dx^2 + c_{-1}\dx^2 + \bigO(\dx^3), \label{eq: 1step central -1}\\
L(\bar{s}^n,\bar{s}^{n+1})_{0} &=a_{0}\dx^2 + b_{0}\dx^2 + c_{0}\dx^2 + \bigO(\dx^3), \label{eq: 1step central }\\
  L(\bar{s}^n,\bar{s}^{n+1})_{1} &=- \frac{1}{4}(1-\lambda)\lambda^2\dx^2 s_{xx}(t^n,x_{1}) + a_{1}\dx^2 + b_{1}\dx^2 + c_{1}\dx^2 + \bigO(\dx^3).
                     \label{eq: 1step central 1}
\end{align}
\end{subequations}
Note that the error that we make when switching from MUSCL to Trapezoidal on cell $-1$ is to leading order the
negative of the error that we make when switching back from Trapezoidal to MUSCL on cell $1$.

  \subsubsection{Numerical results for MUSCL-Trap}
  \label{subsec: num results MUSCL-Trap 1d}

  In \cite{May_Berger_explimpl}, the authors presented numerical results for the model problem shown in
  figure \ref{Fig: 1d model problem}, which we will start with as well to keep the presentation mostly self-contained.
  
  {\textbf{Test 1:}} Consider the model problem shown in figure \ref{Fig: 1d model problem}.
  One cut cell is located in the interval $[0.5,0.5+\alpha \dx]$. The overall
  domain length is given by $[0,1+\alpha \dx]$. Periodic boundary conditions are used. The test function is
  $\sin\left( \frac{2 \pi (x+0.36)}{1 + \alpha \dx}  \right)$, $\alpha = 10^{-4}$, and final time is
  $T =  (1+\alpha \dx)$ so that the test function is back to its original position.
  The $L^1$ error has been normalized to account for the changing domain length.
  We solve $s_t + s_x = 0$ with $\lambda = 0.8$, independent of $\alpha$.
  The results for the one step error and the error at time $T$ are shown in Table \ref{Table: MUSCL-Trap 1d old model problem}.

  \begin{table}
    \begin{center}
    {\small
    \caption{Result of \textbf{Test 1}: MUSCL-Trap for the model problem in Fig. \ref{Fig: 1d model problem}.}
\label{Table: MUSCL-Trap 1d old model problem}
\begin{tabular}{crcccccc}
\hline 
Final time & $\dx$ & $L^1$ error & order & $L^{\infty}$ error & order \\
\hline
\\[-0.2cm]
1 step error& 1/160 	 & 2.95e-06 & -- & 5.52e-04 & -- \\
& 1/320 	 & 3.65e-07 & 3.01 & 1.36e-04 & 2.02 \\
& 1/640 	 & 4.55e-08 & 3.01 & 3.36e-05 & 2.01 \\ [0.1cm]
1 period error     & 1/160 	 & 5.68e-05 & -- & 3.51e-04 & -- \\
& 1/320 	 & 1.48e-05 & 1.94 & 7.41e-05 & 2.24 \\
& 1/640 	 & 3.77e-06 & 1.97 & 1.77e-05 & 2.07 \\[0.05cm]

\hline
\end{tabular}
}   
\end{center}
\end{table}

For the one step error in the $L^{\infty}$ norm we observe second-order convergence,
as expected. The third-order
convergence in the $L^{1}$ norm can be explained by a simple counting argument: let $M$ denote the number of Cartesian cells. Then, there are
$M\text{-}3$ cells (out of $M\text{+}1$ total cells) with an error of size $\bigO(\dx^3)$ and only 4 cells with an error of size $\bigO(\dx^2)$.
This results in the third-order convergence in the $L^{1}$ norm.

The results for the error at time $T$ are more surprising: we observe full second-order convergence in the $L^1$ and $L^{\infty}$ norm
despite the one step error only converging with second order in $L^{\infty}$.
However, there is a reasonable explanation for this behavior. Let us consider error accumulation from a Lagrangian view point: let us
fix one cell, e.g., the cell that contains the peak of the sine curve, put a tracer there, and follow the tracer during the simulation.
If we assume that our scheme is stable (and that error propagates with norm smaller/equal than 1),
the error at time $T$ essentially corresponds to the sum of all the one step errors of all the cells that our tracer has passed.
Given that most cells have a one step error of third order and that there are only 4 isolated cells with a one step error of second order,
this consideration actually implies to expect a second-order error on all cells at time $T$. This is consistent with the numerical
results.

This consideration motivates the following new test, \textbf{Test 2}, which we expect to be more challenging:
Instead of using only one cut cell, we double the number of cut cells as we halve the grid size $\dx$. Then, the number
of cells with one step error $\bigO(\dx^2)$ increase with $\bigO(\frac{1}{\dx})$ and should affect the convergence order at time $T$.

\begin{figure}[b]
  \centering
  \begin{tikzpicture}[scale=0.75]
    \node[] at (-2,0) {{\small $L=1,K=6$:}};
    \draw (0,0) -- (12.4,0);
    \node[] at (1,0.2) {{\footnotesize $h$}};    
    \draw (0,-0.2) -- (0,0.2);
    \draw (2,-0.2) -- (2,0.2);
    \draw (4,-0.2) -- (4,0.2);
    \draw[azure] (6,-0.2) -- (6,0.2);
    \draw[azure] (6.4,-0.2) -- (6.4,0.2);
    \draw (8.4,-0.2) -- (8.4,0.2);
    \draw (10.4,-0.2) -- (10.4,0.2);
    \draw (12.4,-0.2) -- (12.4,0.2);
     \node[] at (-2,-1) {{\small $L=2,K=6$:}};
    \draw (0,-1) -- (12.4,-1);
    \node[] at (0.5,-0.8) {{\footnotesize $h/2$}};    
    \draw (0,-1.2) -- (0,-0.8);
    \draw (1,-1.2) -- (1,-0.8);
    \draw (2,-1.2) -- (2,-0.8);
    \draw[azure] (3,-1.2) -- (3,-0.8);
    \draw[azure] (3.2,-1.2) -- (3.2,-0.8);
    \draw (4.2,-1.2) -- (4.2,-0.8);
    \draw (5.2,-1.2) -- (5.2,-0.8);
    \draw[line width = 0.5mm] (6.2,-1.2) -- (6.2,-0.8);
    \draw (7.2,-1.2) -- (7.2,-0.8);
    \draw (8.2,-1.2) -- (8.2,-0.8);
    \draw[azure] (9.2,-1.2) -- (9.2,-0.8);
    \draw[azure] (9.4,-1.2) -- (9.4,-0.8);
    \draw (10.4,-1.2) -- (10.4,-0.8);
    \draw (11.4,-1.2) -- (11.4,-0.8);
    \draw (12.4,-1.2) -- (12.4,-0.8);
    %
     \node[] at (-2,-2) {{\small $L=4,K=6$:}};
    \draw (0,-2) -- (12.4,-2);
    \node[] at (0.25,-1.8) {{\footnotesize $h/4$}};    
    \draw (0,-2.2) -- (0,-1.8);
    \draw (0.5,-2.2) -- (0.5,-1.8);
    \draw (1,-2.2) -- (1,-1.8);
    \draw[azure] (1.5,-2.2) -- (1.5,-1.8);
    \draw[azure] (1.6,-2.2) -- (1.6,-1.8);
    \draw (2.1,-2.2) -- (2.1,-1.8);
    \draw (2.6,-2.2) -- (2.6,-1.8);
    \draw[line width = 0.5mm] (3.1,-2.2) -- (3.1,-1.8);
    \draw (3.6,-2.2) -- (3.6,-1.8);
    \draw (4.1,-2.2) -- (4.1,-1.8);
    \draw[azure] (4.6,-2.2) -- (4.6,-1.8);
    \draw[azure] (4.7,-2.2) -- (4.7,-1.8);
    \draw (5.2,-2.2) -- (5.2,-1.8);
    \draw (5.7,-2.2) -- (5.7,-1.8);
    \draw[line width = 0.5mm] (6.2,-2.2) -- (6.2,-1.8);
    \draw (6.7,-2.2) -- (6.7,-1.8);
    \draw (7.2,-2.2) -- (7.2,-1.8);
    \draw[azure] (7.7,-2.2) -- (7.7,-1.8);
    \draw[azure] (7.8,-2.2) -- (7.8,-1.8);
    \draw (8.3,-2.2) -- (8.3,-1.8);
    \draw (8.8,-2.2) -- (8.8,-1.8);
    \draw[line width = 0.5mm] (9.3,-2.2) -- (9.3,-1.8);
    \draw (9.8,-2.2) -- (9.8,-1.8);
    \draw (10.3,-2.2) -- (10.3,-1.8);
    \draw[azure] (10.8,-2.2) -- (10.8,-1.8);
    \draw[azure] (10.9,-2.2) -- (10.9,-1.8);
    \draw (11.4,-2.2) -- (11.4,-1.8);
    \draw (11.9,-2.2) -- (11.9,-1.8);
    \draw (12.4,-2.2) -- (12.4,-1.8);
    \end{tikzpicture}
  \caption{New 1d model problem: The grid consists of $L$ blocks with $K+1$ cells that each contain $K$ cells of length $H$ (here $H=h,\frac h 2, \frac h 4$) and one cut cell of length $\alpha H$ in the middle.
   When the mesh is refined by a factor of 2, we keep $K$ (and $\alpha$) fixed but double $L$. Analogous to figure \ref{Fig: 1d model problem}, we refer to cut cells as (cells of type) 0, to the left Cartesian neighbors of cut cells as (cells of type) -1, and to the right Cartesian neighbor as (cells of type) 1. }
\label{Fig: new 1d model problem}
\end{figure}
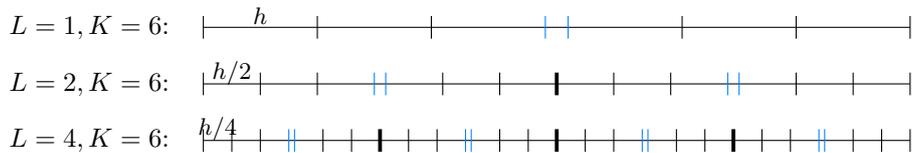

\textbf{Test 2:} The test setup is shown in figure \ref{Fig: new 1d model problem}. We use $K=40$ and $L=4$ on the coarsest mesh.
  The overall
  domain length is given by $[0,1+4\alpha \dx_0]$ with $\dx_0$ denoting the coarsest mesh (in our case $\dx_0=\frac{1}{160}$).
  Different to \textbf{Test 1}, the domain length stays constant.
  The test function is
  $\sin\left( \frac{2 \pi (x+0.36)}{1 + 4\alpha \dx_0}  \right)$, $\alpha = 10^{-4}$, and
  $T = (1+4\alpha \dx_0)$ so that the test function is back to its original position.
  The results for the one step error and the error at time $T$ are shown in Table \ref{Table: MUSCL-Trap 1d new model problem}.

  \begin{table}
    \begin{center}
 {\small
    \caption{Result of \textbf{Test 2}: Error for MUSCL-Trap for the model problem shown in Fig. \ref{Fig: new 1d model problem}.}
\label{Table: MUSCL-Trap 1d new model problem}
\begin{tabular}{crcccccc}
\hline 
Final time & $\dx$ & $L^1$ error & order & $L^{\infty}$ error & order \\
\hline
\\[-0.2cm]
1 step error & 1/160 	 & 7.81e-06 & -- & 6.86e-04 & -- \\
& 1/320 	 & 2.22e-06 & 1.82 & 1.67e-04 & 2.04 \\
& 1/640 	 & 5.40e-07 & 2.04 & 4.31e-05 & 1.95 \\[0.1cm]
1 period error   & 1/160 	 & 4.76e-05 & -- & 5.10e-04 & -- \\
& 1/320 	 & 1.15e-05 & 2.05 & 1.06e-04 & 2.27 \\
& 1/640 	 & 2.90e-06 & 1.98 & 2.27e-05 & 2.23 \\  [0.05cm]
\hline
\end{tabular}  }   
\end{center}
\end{table}

Despite the one step error being only of second order in both the $L^1$ and $L^{\infty}$ norm, the error after one period also converges
with second order both in $L^1$ and $L^{\infty}$. This behavior of the error accumulation behaving differently on non-uniform meshes has been observed and analyzed before
\cite{Wendroff_White}. We will follow the proof from \cite{Wendroff_White} to show second order accuracy for MUSCL-Trap in theorem 
\ref{thm: 2nd order MUSCl-Trap} below. Similar ideas have also been used to show the second order accuracy of the $h$-box method \cite{Berger_Helzel_Leveque_2003}.

Different to \textbf{Test 1}, we observe for the one step error measured in the $L^1$ norm, a convergence order of 2 (compared to 3 before).
This can be explained as follows: let $e_i =  S_i^n - \bar{s}_i^n$ denote the error on cell $i$, let $K\text{+}1$ be the number of cells per block and let there be
$L$ blocks. As before, $M$ denotes the total number of Cartesian cells. This implies
$M+L = L \cdot (K+1)$. Then, the $L^1$ error can be computed as
(with $c$ denoting a generic constant)
\begin{align*}
\sum_{\text{cells}} \dx_i \cdot \abs{e_i} 
&= \sum_{\substack{\text{`normal'} \\ \text{cells}}} \dx \cdot \bigO(\dx^3) + \sum_{\substack{\text{cells} \\ -1,1,2}} \dx \cdot \bigO(\dx^2) 
+ \sum_{\text{cell } 0} \alpha \dx \cdot \bigO(\dx^2) \\
&= (M-3L) \cdot c\dx^4 + 3L \cdot c\dx^3 + L \cdot \alpha c\dx^3.
\end{align*}
Using $L = \frac{M}{K}$ and $M = \mathcal{O}(\frac{1}{\dx})$, we get
\begin{align*}
\sum_{\text{cells}} \dx_i \cdot \abs{e_i} 
&= \left( 1 - \frac{3}{K} \right) \cdot c\dx^3
+ 3 \frac{1}{K} \cdot c\dx^2 + \frac{1}{K} \cdot \alpha c\dx^2.
\end{align*}
Therefore, for keeping $K$ (the number of cells per block) fixed, we get an $L^1$ error of $\bigO(\dx^2)$. This theoretical consideration
coincides with our numerical observation.
Note that in \textbf{Test 1}, $K$ increased with $\bigO(\frac{1}{\dx})$, explaining the better convergence rate of 3 in the $L^1$ norm for the one step error.

  \subsubsection{Proof for MUSCL-Trap being second-order accurate}
  \label{subsec: proof Wendroff white}

For the proof of second-order convergence, we require stability of the method. Examining this stability mathematically is very challenging
due to using a
non-uniform mesh and a mixed scheme. Applying, for example, von Neumann stability analysis is not feasible.
We therefore formulate this as assumption \ref{Ass: Stab} in the proof below. In our numerical tests, we found that this was true for the $L^1$ and $L^{\infty}$ norm.
\begin{theorem}\label{thm: 2nd order MUSCl-Trap}
  Let Ass. \ref{Ass: Stab} below hold true. Consider the MUSCL-Trap scheme given by \eqref{eq: mixed scheme in 1d}, but assume that slopes are computed by means of an unlimited forward difference quotient. Let $0 < \lambda \le 1$ independent of $\alpha$. Then
  the scheme is second-order accurate with respect to the norm used in Ass. \ref{Ass: Stab}
  for the linear advection equation for model problem
  \ref{Fig: new 1d model problem} for sufficiently smooth initial data $s_0$.
  \end{theorem}

  \begin{proof}
    For the proof, we use an idea that goes back to Wendroff and White \cite{Wendroff_White}, and which has also been used to show second-order accuracy of the $h$-box method \cite{Berger_Helzel_Leveque_2003}:
Assume that we are able to construct a grid function with cell averages $\bar{w}_i^n$ such that (i) the new grid function $\bar{w}_i^n$ is sufficiently close to $\bar{s}_i^n$ and (ii) the one step error is of third order for all cells. To be precise, $\bar{w}_i^n$ is supposed to satisfy
    \begin{subequations}\label{eq: condition WW}
  \begin{alignat}{2}
&\text{(i)} \  &&\bar{w}^n_i=\bar{s}^n_i + \bigO(\dx^2) \quad \forall \: i,n,  \label{eq: condition WW 1} \\
&\text{(ii)} \  &&\abs{L(\bar{w}^{n},\bar{w}^{n+1})_i} =  \bigO(\dx^3)  \quad \forall \: i,n. \label{eq: condition WW 2}
\end{alignat}
\end{subequations}
 \begin{assumption}\label{Ass: Stab}
    We assume that the MUSCL-Trap scheme is stable with respect to the norm $\lVert \cdot \rVert$ in the
    following sense: 
    there holds
    \[
      \lVert \Phi(S^{n-1} - \bar{w}^{n-1},S^{n} - \bar{w}^n)\rVert \le \lVert S^{n-1} - \bar{w}^{n-1}\rVert + \mathcal{O}(h^3) \quad \: \forall \: 1 \le n \le N
      \]
for the error propagation.
\end{assumption}
For the error grid function $E^n = S^n - \bar{w}^{n}$ there holds by linearity of $\Phi$
\begin{align*}
  S^n - \bar{w}^{n} &= \Phi(S^{n-1},S^{n}) - \bar{w}^{n}
  = \Phi(E^{n-1}+\bar{w}^{n-1},E^{n}+\bar{w}^{n}) - \bar{w}^{n}\\
  &= \Phi(E^{n-1},E^{n}) + \Phi(\bar{w}^{n-1},\bar{w}^{n}) - \bar{w}^{n}
  = \Phi(E^{n-1},E^{n}) + L(\bar{w}^{n-1},\bar{w}^{n}).
\end{align*}
  From Ass. \ref{Ass: Stab} and assumption \eqref{eq: condition WW 2}, we can conclude
    \begin{align*}
      \lVert S^N - \bar{w}^N\rVert
      &\le \lVert \Phi(S^{N-1} - \bar{w}^{N-1},S^{N} - \bar{w}^N) \rVert + \lVert L(\bar{w}^{N-1},\bar{w}^{N}) \rVert \\
      &\le \lVert S^{N-1} - \bar{w}^{N-1} \rVert + \bigO(\dx^3)
    \le \lVert S^0 - \bar{w}^0\rVert + \frac{T}{\dt}\bigO(\dx^{3})=\bigO(\dx^2)
\end{align*}  
with $N = \frac{T}{\dt}$ and $\dt = \bigO(\dx)$. In other words: we have `normal' error accumulation with respect to the new solution $\bar{w}^{n}$.
Together with property \eqref{eq: condition WW 1} we get by means of the triangle inequality
\begin{equation*}
\lVert S^N-\bar{s}^N\rVert \le \lVert S^N-\bar{w}^N\rVert + \lVert \bar{w}^N-\bar{s}^N\rVert = \bigO(\dx^2) + \bigO(\dx^2) = \bigO(\dx^2),
\end{equation*}
which implies global, second-order convergence with respect to the true solution $\bar{s}$.

It remains to find such a suitable grid function $\bar{w}^{n}$. We note that this is the essence of this proof.

\textbf{Case $\alpha=1$:} We first consider the case of a uniform mesh. The main error source that we like to examine is the
switch in time stepping. Note though that different to \eqref{eq: 1step central}, we now assume that forward differences
are used. Therefore, the actual error terms look slightly different and are given by
  \begin{align*}
L(\bar{s}^{n},\bar{s}^{n+1})_{-1} &=
-\frac{1}{4}\lambda^3\dx^2 s_{xx}(t^n,x_{-1}) + \bigO(\dx^3), \quad
L(\bar{s}^{n},\bar{s}^{n+1})_{0} = \bigO(\dx^3), \\
L(\bar{s}^{n},\bar{s}^{n+1})_{1} &=
  \frac{1}{4}\lambda^3\dx^2 s_{xx}(t^n,x_{1}) + \bigO(\dx^3).
    \end{align*}
We define
\begin{equation}\label{eq: def w alpha=1}
  \bar{w}^n_i = \bar{s}^n_i + \gni \quad \text{ with } \quad
  \gni= \begin{cases}
    - \frac{1}{2}\lambda^2 \dx^2 s^n_{xx}(t^n,x_0) & \text{ for } i=0,\\
    0 & \text{ otherwise}.\\
    \end{cases}
  \end{equation}
  A direct computation then shows $L(\gamma^n,\gamma^{n+1})_{0} = \bigO(\dx^3)$ as well as
  \[
    L(\gamma^n,\gamma^{n+1})_{-1} = - \frac{\lambda}{2} \gnn + \bigO(\dx^3), \quad
    L(\gamma^n,\gamma^{n+1})_{1} = \frac{\lambda}{2} \gnn + \bigO(\dx^3).
    \]
This implies by linearity
  \[
L(\bar{w}^n,\bar{w}^{n+1})_i = \bigO(\dx^3) \quad \forall \: i.
\]
Therefore, assumptions \eqref{eq: condition WW 1} and \eqref{eq: condition WW 2}
are satisfied.\\
  
\textbf{Case $\alpha<1$:} In this case we need to address all four error sources.
First, a direct computation shows that for forward differences there holds with $\beta = \frac{1+\alpha}{2}$
\begin{subequations}\label{eq: 1step forward}
\begin{align}
L(\bar{s}^n,\bar{s}^{n+1})_{-1}&=\frac{\lambda}{4}\Big(-\lambda^2 + (1-\beta) -\frac{1}{2\beta}\frac{1}{6}(\alpha^2-1)\Big)\dx^2s^n_{xx}(t^n,x_{-1}) + \bigO(\dx^3),\label{eq: 1step forward -1}\\
L(\bar{s}^n,\bar{s}^{n+1})_{0} &=\bigO(\dx^3), \label{eq: 1step forward 0}\\
L(\bar{s}^n,\bar{s}^{n+1})_{1} &=\frac{\lambda}{4}\Big(\lambda^2 + (\beta-1) +\frac{1}{2\beta}\frac{1}{6}(\alpha^2-1)  \Big)\dx^2s^n_{xx}(t^n,x_1) + \bigO(\dx^3).\label{eq: 1step forward 1}
\end{align}
\end{subequations}
Note in particular that for this setup, the one step error on the small cell is surprisingly of third order as errors cancel nicely.
We define
\begin{equation}\label{eq: def w alpha != 1}
  \bar{w}^n_i = \bar{s}^n_i + \gni \quad \text{ with } \quad
  \gni= \begin{cases}
    \gnn & \text{ for } i=0,\\
    0 & \text{ otherwise},\\
    \end{cases}
  \end{equation}
  with
    \begin{equation}\label{eq: def gnn}
\gnn=\frac{\beta}{2}\Big(-\lambda^2+(1-\beta)-\frac{1}{2\beta}\frac{1}{6}(\alpha^2-1) \Big) \dx^2 s^n_{xx}(t^n,x_0) .
\end{equation}
This results in $L(\gamma^n,\gamma^{n+1})_{0} = \bigO(\dx^3)$ as well as in
\[
    L(\gamma^n,\gamma^{n+1})_{-1} = - \frac{\lambda}{2\beta} \gnn + \bigO(\dx^3), \quad
    L(\gamma^n,\gamma^{n+1})_{1} = \frac{\lambda}{2\beta} \gnn + \bigO(\dx^3),
  \]
  which implies the claim.
\end{proof}

\begin{remark}
  The advantage of forward differences compared to central differences lies in the reduced coupling of cells.
  For central differences, we have not been able to find suitable coefficients $\gamma^n$ when using periodic boundary conditions.
  But the numerical results in section \ref{subsec: num results MUSCL-Trap 1d} imply that the same order of convergence holds true
  when one employs central differences instead of forward ones.
\end{remark}

\begin{remark}
  We note that we addressed two different error sources in this proof: the error
  caused by the switch in the time stepping scheme and the error caused by the irregular length of the cut cell.
  The reason that the former error does
  not accumulate is simply that we make an error of size $-\frac{\lambda^3}{4}\dx^2 s_{xx}(t^n,x_{-1})$ on cell $-1$ and an error of size $\frac{\lambda^3}{4}\dx^2 s_{xx}(t^n,x_{1})$ on cell $1$. In other words: the error that we made on cell $-1$ cancels to leading order with the error that we make two cells later. 
  \end{remark}

    Finally, we verify our results numerically. We repeat \textbf{Test 2} using forward difference quotients but this time we compare our discrete solution $S_i^n$ to our new
    grid function $\bar{w}$ given by \eqref{eq: def w alpha != 1} and \eqref{eq: def gnn}. The result is given in Table \ref{Table: MUSCL-Trap 1d new model problem new grid fct} and shows the
    expected convergence rates.

    \begin{table}
      \begin{center}
 {\small
   \caption{Result of \textbf{Test 2}: Error for MUSCL-Trap using forward difference quotients for the model problem shown in figure \ref{Fig: new 1d model problem}. 
     The new grid function $\bar{w}$ given by \eqref{eq: def w alpha != 1} is used for initialization and error computation.}
\label{Table: MUSCL-Trap 1d new model problem new grid fct}
\begin{tabular}{crcccccc}
\hline 
Final time & $\dx$ & $L^1$ error & order & $L^{\infty}$ error & order \\
\hline
\\[-0.2cm]
  1 step error
%
  & 1/160 	 & 1.74e-06 & -- & 5.23e-06 & -- \\
& 1/320 	 & 2.24e-07 & 2.95 & 6.29e-07 & 3.06 \\
& 1/640 	 & 2.85e-08 & 2.98 & 8.21e-08 & 2.94 \\[0.1cm]

  1 period error
& 1/160 	 & 3.74e-04 & -- & 5.90e-04 & -- \\
& 1/320 	 & 9.28e-05 & 2.01 & 1.46e-04 & 2.02 \\
& 1/640 	 & 2.32e-05 & 2.00 & 3.71e-05 & 1.98 \\[0.05cm]  
\hline
\end{tabular}  }   
\end{center}
\end{table}

\FloatBarrier

\subsection{New variants of mixed explicit implicit schemes}

Our theoretical and numerical considerations above imply that the mixed scheme MUSCL-Trap is second-order accurate in 1d, despite only having a 
second-order one step error. 
The numerical results in 2d though for MUSCL-Trap presented in \cite{May_Berger_explimpl} showed
convergence rates between 1.3 and 1.6 measured in the $L^{\infty}$ norm. This implies that in 2d
the error {\em does accumulate} to some extent. Therefore, we cannot expect to be able to find a new grid function
$\bar{w}^n$ with properties (i) and (ii) in 2d.

In order to find a scheme with a third-order one step error, we would need to address all four error sources. As the error sources 2(a)-(c) are very difficult to examine in 2d (due to varying sizes and shapes of cut cells) and are generally shared by most cut cell schemes, we focus here on the error caused by
switching the time stepping scheme.
Further, we expect the error for switching the time stepping scheme to have a very different effect in 1d and 2d: in 1d, characteristics are typically set up in a way that an error caused by switching from explicit to implicit is followed very briefly afterwards by an error caused by switching from implicit back to explicit.
In 2d however, see the setup of the ramp test below in figure \ref{Fig: switching scheme in 2d}, this effect of switching back and forth might not be there. Instead there can be characteristics that mainly go through transition cells.

\subsubsection{Developing mixed time stepping schemes with better transition error}

We will discuss the issue of reducing the error by switching the time stepping scheme in a sightly more general setting.
For finite volume schemes, one typically uses the conservative update formula
\begin{equation}\label{eq: cons update}
S_i^{n+1} = S_i^n - \frac{\dt}{\dx} \left[ F_{i+1/2} - F_{i-1/2} \right]
\end{equation}
with $F_{i\pm 1/2}$ being an appropriate approximation to the flux at the edge $x_{i\pm 1/2}$ during the time step.
In the case of the linear advection equation \eqref{eq: lin adv} (with WLOG $u=1$), we try to approximate
\begin{equation}\label{eq: flux to be approx}
F_{i+1/2}^{\text{true}} = \frac{1}{\dt} \int_{t_n}^{t_{n+1}} s(t,x_{i+1/2}) \: dt.
  \end{equation}

  To only focus on the temporal error, we will reinterpret time stepping schemes as quadrature formulae for approximating
  the integral in \eqref{eq: flux to be approx}, assuming that we know all values at the edges $x=x_{i\pm 1/2}$ that we need.
  To give an example: using explicit Euler in time would correspond to approximating $F_{i+1/2}^{\text{true}}$ by the point evaluation
  $s(t^n,x_{i+1/2})$. For MUSCL and Trapezoidal time stepping we have the following results:
  \begin{itemize}
  \item MUSCL: The corresponding time stepping scheme is the {\em explicit} midpoint rule, which corresponds to
    the following approximation of $F_{i+1/2}^{\text{true}}$:
    \[
F_{i+1/2}^{(M)} = s(t^n,x_{i+1/2}) + \frac{\dt}{2} \: s_t(t^n,x_{i+1/2}).
    \]
    This results in
    \[
F_{i+1/2}^{\text{true}} - F_{i+1/2}^{(M)} = \frac{\dt^2}{6} s_{tt}(t^n,x_{i+1/2}) +  \bigO(\dt^3).
\]
\item Trapezoidal: The Trapezoidal time stepping scheme approximates $F_{i+1/2}^{\text{true}}$ by
  \[
F_{i+1/2}^{(T)} = \frac{1}{2} \left[ s(t^n,x_{i+1/2}) + s(t^{n+1},x_{i+1/2}) \right].
\]
This results in the error
        \[
       F_{i+1/2}^{\text{true}} -  F_{i+1/2}^{(T)} = - \frac{\dt^2}{12} s_{tt}(t^n,x_{i+1/2}) +  \bigO(\dt^3).
\]
\end{itemize}
Therefore, assuming that the reconstruction in space is done sufficiently accurately, when using Trapezoidal rule time stepping, we 
approximate $F_{i+1/2}^{\text{true}}$ with an error of $\bigO(\dt^2)$. However, when we use Trapezoidal {\em on both edges},
we make to leading order the same (systematic) error for the approximation of
$F_{i-1/2}^{\text{true}}$. When applying the update formula \eqref{eq: cons update}, the leading order error term cancels and there holds
\begin{multline*}
  \left[ F_{i+1/2}^{\text{true}} - F_{i-1/2}^{\text{true}} \right]
  -  \left[ F_{i+1/2}^{(T)} - F_{i-1/2}^{(T)} \right]\\
  = -\frac{\dt^2}{12} s_{tt}(t^n,x_{i+1/2}) + \frac{\dt^2}{12} s_{tt}(t^n,x_{i-1/2}) +  \bigO(\dt^3)
  = \bigO(\dt^2\dx) + \bigO(\dt^3).
  \end{multline*}
  Therefore, for $\dx = \bigO(\dt)$, 
  we get an error
of order $\bigO(\dt^3)$ (for the time approximation).
%
However, if we use the explicit midpoint rule (
$\hat{=}$ MUSCL) for $F_{i-1/2}$ and Trapezoidal time stepping
for $F_{i+1/2}$, we get
\begin{multline*}
  \left[ F_{i+1/2}^{\text{true}} - F_{i-1/2}^{\text{true}} \right]
  - \left[ F_{i+1/2}^{(T)} - F_{i-1/2}^{(M)} \right] \\
  = -\frac{\dt^2}{12} s_{tt}(t^n,x_{i+1/2}) - \frac{\dt^2}{6} s_{tt}(t^n,x_{i-1/2}) +  \bigO(\dt^3)
  = \bigO(\dt^2).
  \end{multline*}
  This is the transition error that we saw before.

Note that these considerations have another important implication: the {\em seemingly} easiest way to improve the transition error that we observed for MUSCL-Trap
would be to just keep MUSCL and use a third-order implicit scheme. Or to keep the Trapezoidal time stepping scheme and upgrade the explicit time stepping scheme
to third order. However, both versions would {\em not} show the result one hoped for. The easiest way to see this is by combining, for example,
the MUSCL flux $F_{i-1/2}^{(M)}$ with the true flux $F_{i+1/2}^{\text{true}}$. Then, when applying the conservative update
formula \eqref{eq: cons update}, we would be left with an error of $\bigO(\dt^2)$ despite the scheme being apparently more accurate than when applying MUSCL
fluxes on both edges.

Therefore, the two options to get a scheme with a third-order one step error
(with respect to the time error) are to (a) either use both an explicit and an implicit time stepping scheme
that approximate $F_{i\pm 1/2}^{\text{true}}$ with third order or (b)
to find a combination of a second-order explicit and a second-order implicit time stepping scheme such that the leading order error terms match.

Besides approach (a) being more expensive, third-order time stepping schemes also typically involve several stages, which makes the coupling
very complicated.
We therefore follow approach (b): we fix the implicit Trapezoidal rule and look for a new explicit time stepping scheme. We will consider two different modifications of the MUSCL scheme.

\subsubsection{The MUSCLmod-Trap scheme}

The first approach is to change the MUSCL scheme given in \eqref{eq: MUSCL 1d} to use
\[
F_{i+1/2}^{n+1/2,M\text{mod}} = u \left( S_i^n + (1 - \lambda) S_{i,x}^n \frac{\dx}{2} - \frac{1}{4} \lambda(1-\lambda)h^2 S_{i,xx}^n \right).
\]
The idea is to account for the error term $\frac{\dt}{4}s_{tt}(t^n,x_{1+1/2})$ that we saw above when switching between Trapezoidal and explicit midpoint rule but the details slightly vary as we generally do not have the exact values on the edges. The precise formulation simply comes out of the error analysis for combining MUSCL with Trapezoidal.
The truncation error analysis shows that
\begin{itemize}
\item using MUSCLmod on an equidistant mesh as fully explicit scheme, there holds
  $L(\overline{s}_i^{n+1}) = \mathcal{O}(\dx^3)$;
\item using MUSCLmod in a mixed MUSCLmod-Trap scheme (with $\alpha=1$), there holds
   $L(\overline{s}_i^{n+1}) = \mathcal{O}(\dx^3)$.
 \end{itemize}
 A von Neumann stability analysis for the new MUSCL version shows a CFL condition of $0 < \lambda \le 1$.

 In table \ref{Table: MUSCLmod-Trap  1d new model problem} we show the result for the new mixed MUSCLmod-Trap scheme for \textbf{Test 2}, but using $\alpha=1$, i.e., we switch the time stepping between explicit and implicit at the corresponding cells but do that on an equidistant mesh to avoid the error sources 2(a)-(c). As expected we observe third order for the one step error and second error for the error at time $T$. We use a standard second-order accurate difference quotient to evaluate the second derivatives.

 \begin{table}[h]
   \begin{center}
{\small
    \caption{Result of \textbf{Test 2}: Error for MUSCLmod-Trap for model problem \ref{Fig: new 1d model problem} with $\alpha=1$ (i.e., equidistant mesh but including change in time stepping scheme).}
\label{Table: MUSCLmod-Trap  1d new model problem}
\begin{tabular}{crcccc}
\hline 
Final time & $\dx$ & $L^1$ error & order & $L^{\infty}$ error & order \\
\hline
\\[-0.2cm]
1 step error   & 1/160 	 & 9.06e-07 & -- & 6.65e-06 & -- \\
& 1/320 	 & 1.18e-07 & 2.95 & 8.21e-07 & 3.02 \\
& 1/640 	 & 1.49e-08 & 2.98 & 1.05e-07 & 2.96 \\[0.1cm]
1 period error    & 1/160 	 & 1.79e-04 & -- & 2.78e-04 & -- \\
& 1/320 	 & 4.51e-05 & 1.99 & 6.94e-05 & 2.00 \\
& 1/640 	 & 1.13e-05 & 2.00 & 1.73e-05 & 2.00 \\ [0.05cm]
\hline
\end{tabular}  }   
\end{center}
\end{table}

\subsubsection{The MPRKC-Trap scheme}

For the second variant, consider the 
{\em explicit} Trapezoidal rule, also known as standard second-order SSP RK (strong stability preserving Runge-Kutta) scheme \cite{Gottlieb_Shu}, and given by (for the ODE $y_t = g(y(t))$): 
\[
y^{(1)} = y^n + \dt g(y^n), \quad y^{n+1} = y^n + \frac{\dt}{2} \left[ g(y^n) + g(y^{(1)}) \right].
\]
If we compare that with the implicit Trapezoidal rule, we see that there is a $\mathcal{O}(\dt^2)$ difference due to using a second-order predictor step $y^{(1)}$ instead of $y^{n+1}$. Therefore, to get an explicit scheme that matches the leading order error term of $F_{i+1/2}^{(T)}$ up to third order,
we need to replace the
predictor step (computation of $y^{(1)}$), which currently uses explicit Euler, by a more accurate predictor. We choose the explicit midpoint rule for this purpose. Together with a suitable space discretization (the new predictor step will correspond to the MUSCL scheme), this results in a new explicit finite volume scheme, which
we will present now. Note that by construction this new scheme will show a third-order one step transition
error when coupled to (implicit) Trapezoidal time stepping with slope reconstruction.

The new explicit scheme, which we call MPRKC (for MUSCL Predictor RK Corrector) scheme, is given by
\begin{subequations}\label{eq: MPRKC}
  \begin{align}
    S^{(1)}_i&=S^n_i-\frac{\dt}{\dx}\Big(F^{n+1/2,M}_{i+1/2} - F^{n+1/2,M}_{i+1/2}  \Big),  \label{predictor} \\
    S^{n+1}_i&=S^n_i-\frac{\dt}{\dx} \Big(F^{n+1/2,ET}_{i+1/2}-F^{n+1/2,ET}_{i-1/2}\Big),  \label{corrector}
\end{align}
\end{subequations}
with
\begin{equation}\label{eq: flux explicit}
F^{n+1/2,ET}_{i+1/2}=\frac{u}{2} \Big(S^n_i+S^n_{i,x}\frac{\dx}{2} + S^{(1)}_i+S^{(1)}_{i,x}\frac{\dx}{2} \Big).
\end{equation}
\begin{remark}
  In terms of cost, the new scheme is roughly twice as expensive as MUSCL
  for taking one time step. Compared to using the standard second-order SSP RK scheme
  with slope reconstruction in space, the cost is roughly the same. We just exploit the information from the slope reconstruction in stage 1 more carefully.
  \end{remark}

    \begin{lemma}
      The scheme \eqref{eq: MPRKC}-\eqref{eq: flux explicit} is linearly stable for the linear advection equation \eqref{eq: lin adv} on an equidistant mesh under the
      CFL condition $0 < \lambda \le 1.$
      \end{lemma}
\begin{proof}
  The claim follows by means of a von Neumann stability analysis. We analytically deduce a very lengthy and complicated expression
  for the amplification factor $\abs{G}$ (not given here). We then verify numerically that $\abs{G} \le 1$ if and only if $ 0 < \lambda \le 1$.
\end{proof}

We combine the explicit MPRKC scheme with the implicit Trapezoidal scheme \eqref{eq: Trap 1d}
using flux bounding, resulting in MPRKC-Trap. This is sketched in
figure \ref{Fig: MPRKC-Trap}.
Due to MPRKC being a two-stage scheme, the implicit region has become bigger
compared to MUSCL-Trap.
\begin{figure}
  \centering
    \begin{tikzpicture}[scale=0.9,
axis/.style={very thick, line join=miter, ->}]
\draw [axis] (-6.3,-0.3) -- (-6.3,1.4) node(xline)[right] {};
%
%
\node[] at (-6.8,0) {{\small $t^n$}};
\node[] at (-6.8,1.24) {{\small $t^{n+1}$}};
\draw[dotted] (-6.5,0) -- (-3.6,0);
\draw (-5.4,0) -- (5.9,0);
\draw[dotted] (4.1,0) -- (6.9,0);
\draw (-5,-0.1) -- (-5,0.1);
\draw (-4,-0.1) -- (-4,0.1);
\draw (-3,-0.1) -- (-3,0.1);
\draw (-2,-0.1) -- (-2,0.1);
\draw (-1,-0.1) -- (-1,0.1);
\draw (0,-0.1) -- (0,0.1);
\draw (0.5,-0.1) -- (0.5,0.1);
\draw (1.5,-0.1) -- (1.5,0.1);
\draw (2.5,-0.1) -- (2.5,0.1);
\draw (3.5,-0.1) -- (3.5,0.1);
\draw (4.5,-0.1) -- (4.5,0.1);
\draw (5.5,-0.1) -- (5.5,0.1);
\node[] at (-4.5,-0.4) {{\footnotesize $-5$}};
\node[] at (-3.5,-0.4) {{\footnotesize $-4$}};
\node[] at (-2.5,-0.4) {{\footnotesize $-3$}};
\node[] at (-1.5,-0.4) {{\footnotesize $-2$}};
\node[] at (-0.5,-0.4) {{\footnotesize $-1$}};
\node[] at (0.25,-0.4) {{\footnotesize $0$}};
\node[] at (1,-0.4) {{\footnotesize $1$}};
\node[] at (2,-0.4) {{\footnotesize $2$}};
\node[] at (3,-0.4) {{\footnotesize $3$}};
\node[] at (4,-0.4) {{\footnotesize $4$}};
\node[] at (5,-0.4) {{\footnotesize $5$}};
\draw[fill] (-2.5,0) circle (.08cm);
\draw[fill] (-1.5,0) circle (.08cm);
\draw[fill] (-0.5,0) circle (.08cm);
\draw[fill] (0.25,0) circle (.08cm);
\draw[fill] (1,0) circle (.08cm);
\draw[fill] (2,0) circle (.08cm);
 \draw[fill] (3,0) circle (.08cm);
\draw[dotted] (-6.,1.04) -- (-3.6,1.04);
\draw (-5.4,1.04) -- (5.9,1.04);
\draw[dotted] (4.1,1.04) -- (6.9,1.04);
\draw (-5,0.94) -- (-5,1.14);
\draw (-4,0.94) -- (-4,1.14);
\draw (-3,0.94) -- (-3,1.14);
\draw (-2,0.94) -- (-2,1.14);
\draw (-1,0.94) -- (-1,1.14);
\draw (0,0.94) -- (0,1.14);
\draw (0.5,0.94) -- (0.5,1.14);
\draw (1.5,0.94) -- (1.5,1.14);
\draw (2.5,0.94) -- (2.5,1.14);
\draw (3.5,0.94) -- (3.5,1.14);
\draw (4.5,0.94) -- (4.5,1.14);
\draw (5.5,0.94) -- (5.5,1.14);
\draw[fill=white] (-4.5,1.04) circle (.12cm);
\draw[fill=white] (-3.5,1.04) circle (.12cm);
\draw[fill=white] (4,1.04) circle (.12cm);
\draw[fill=white] (5,1.04) circle (.12cm);
\draw[fill=white] (-2.62,0.92) rectangle (-2.38,1.16);
\draw[fill=white] (-1.62,0.92) rectangle (-1.38,1.16);
\draw[fill=white] (-0.62,0.92) rectangle (-0.38,1.16);
\draw[fill=white] (0.13,0.92) rectangle (0.37,1.16);
\draw[fill=white] (0.88,0.92) rectangle (1.12,1.16);
\draw[fill=white] (1.88,0.92) rectangle (2.12,1.16);
\draw[fill=white] (2.88,0.92) rectangle (3.12,1.16);
\draw[line width = 0.7mm,OliveGreen] (-5,-0) -- (-5,1.04);
\draw[line width = 0.7mm,OliveGreen] (-4,-0) -- (-4,1.04);
\draw[line width = 0.7mm,OliveGreen] (-3,-0) -- (-3,1.04);
\draw[line width = 0.7mm,azure] (-2,-0) -- (-2,1.04);
\draw[line width = 0.7mm,azure] (-1,-0) -- (-1,1.04);
\draw[line width = 0.7mm,azure] (0,-0) -- (0,1.04);
\draw[line width = 0.7mm,azure] (0.5,-0) -- (0.5,1.04);
\draw[line width = 0.7mm,azure] (1.5,-0) -- (1.5,1.04);
\draw[line width = 0.7mm,azure] (2.5,-0) -- (2.5,1.04);
\draw[line width = 0.7mm,OliveGreen] (3.5,-0) -- (3.5,1.04);
\draw[line width = 0.7mm,OliveGreen] (4.5,-0) -- (4.5,1.04);
\draw[line width = 0.7mm,OliveGreen] (5.5,-0) -- (5.5,1.04);
\end{tikzpicture}
  \caption{The new MPRKC-Trap scheme: We compute $S_i^{(1)}$ using the MUSCL scheme for cells
    $i \le -2$ and $i \ge 2$ as sketched in figure \ref{Fig: switch scheme: done with expl scheme}. Then we compute
    the explicit flux $F^{n+1/2,ET}$ given by \eqref{eq: flux explicit} for all edges marked in
    green and update the fully explicitly treated cells $i \le -4$ and $i \ge 4$. We combine this using flux bounding with an implicit treatment (edges marked in light blue) of cells $-3,\ldots,3$. Note that now cells $-3$ and $3$ correspond to the transition cells. }
\label{Fig: MPRKC-Trap}
\end{figure}
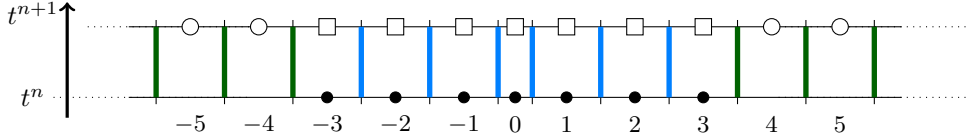
The results for the mixed scheme on an equidistant mesh with $\alpha=1$ but with switching the scheme as indicated in \textbf{Test 2} are shown in table
\ref{Table: MPRKC-Trap 1d new model problem}. As expected, the one step error converges with third order and the error at time $T$ with second order.

\begin{table}
  \begin{center}
{\small
    \caption{Result of \textbf{Test 2}: Error for MPRKC-Trap for model problem \ref{Fig: new 1d model problem} with $\alpha=1$ (i.e., equidistant mesh but including change in time stepping scheme).}
\label{Table: MPRKC-Trap 1d new model problem}
\begin{tabular}{crcccc}
\hline 
Final time & $\dx$ & $L^1$ error & order & $L^{\infty}$ error & order \\
\hline
\\[-0.2cm]
1 step error  
& 1/160 	 & 8.45e-07 & -- & 2.10e-06 & -- \\
& 1/320 	 & 1.07e-07 & 2.98 & 2.59e-07 & 3.01 \\
           & 1/640 	 & 1.36e-08 & 2.98 & 3.32e-08 & 2.97 \\ [0.1cm]
1 period error  & 1/160 	 & 1.81e-04 & -- & 2.79e-04 & -- \\
& 1/320 	 & 4.52e-05 & 2.00 & 6.95e-05 & 2.01 \\
& 1/640 	 & 1.13e-05 & 2.00 & 1.73e-05 & 2.00 \\ [0.05cm]
\hline
\end{tabular}  }   
\end{center}
\end{table}

\section{Mixed explicit implicit schemes in 2d}\label{sec: schemes 2d}

In the following we will discuss mixed explicit implicit schemes in 2d:
we will first briefly introduce the 2d version of the MUSCL-Trap scheme as presented in
\cite{May_Berger_explimpl} and examine its error. Then, we will introduce variants of MUSCLmod-Trap and MPRKC-Trap in 2d. We will conclude with a comparison of the resulting schemes for both a fully Cartesian mesh and a cut cell mesh.

\subsection{The MUSCL-Trap scheme}
In 2d, we solve the linear advection equation
\begin{equation}\label{eq: lin adv 2d}
s_t + u s_x + v s_y = 0.
\end{equation}
Here, $s(t,x,y)$ denotes a scalar field.
For simplicity, we will only consider a
constant velocity field, i.e., $u$ and $v$ are assumed to be constant. This is sufficient to show our main findings.

As explicit scheme, an {\em unsplit}, two-dimensional MUSCL scheme \cite{almgrenBellSzymczak:1996} is used.
For brevity reasons, we will not present the details here.
We note though that this MUSCL scheme uses corner-coupling to capture proper dependencies in 2d
and is therefore stable under the CFL condition $0 < \nu \le 1$ in combination with the time step computation
\begin{equation}\label{eq: compute dt MUSCL}
\dt=\nu\min\left(\frac{\Delta x}{\abs{u}},\frac{\dy}{\abs{v}}\right).
\end{equation}

\begin{figure}
 \centering
 \includegraphics[scale = 0.5]{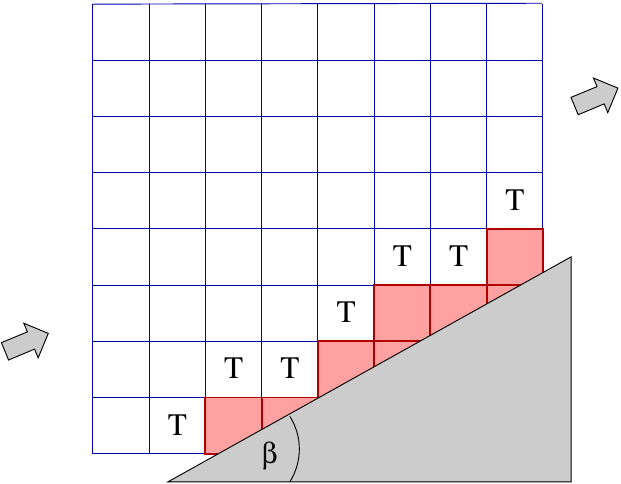}
 \hspace*{0.6cm}
 \includegraphics[scale = 0.5]{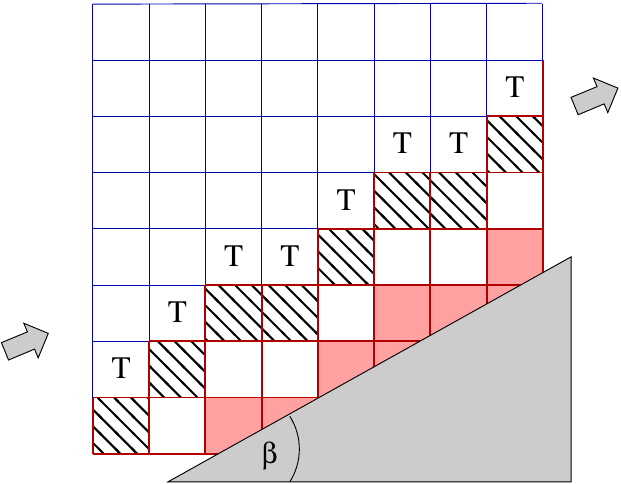}
 \caption{Switching between schemes in 2d: Cut cells are marked by pink color.
   Implicit fluxes are marked by a bold, red line, explicit fluxes by blue.
   \textbf{\em Left:} \textbf{MUSCL-Trap:} `Transition' cells (marked with `T') are full Cartesian cells
   that share an edge with a cut cell. 
   \textbf{\em Right:} \textbf{MPRKC-Trap:} The shaded cells corresponds to the layer of cells closest to the embedded object
   that can be updated using the MUSCL predictor step. Transition cells (marked with `T') have been shifted by 2 cell layers
 compared to MUSCL-Trap. 
 }
 \label{Fig: switching scheme in 2d}
 \end{figure}
%
 Flux bounding is used to couple the explicit MUSCL scheme with an implicit scheme.
 The idea is sketched in the left graphic of figure \ref{Fig: switching scheme in 2d}: Full Cartesian cells that are edge neighbors of cut cells are marked as {\em transition cells}.
 Cut cells are treated fully implicit. Therefore, fluxes between cut cells and transition cells use an implicit scheme. All remaining fluxes, in particular
 fluxes between two transition cells, use an explicit scheme.

 This results in the following two-step algorithm, which is analogous to the 1d situation:
Given $S_{ij}^n$, 
\begin{enumerate}[label=(\roman*)]
\item compute all explicit fluxes using the MUSCL scheme and update all fully explicitly
treated cells to $S_{ij}^{n+1}$;
\item compute all implicit fluxes and update cut cells and transition cells to $S_{ij}^{n+1}$.
\end{enumerate}

As implicit scheme, Trapezoidal time stepping in combination with
a least squares formulation for slope reconstruction in space \cite{Barth_LS_3d,May_Berger_LP} is used. On Cartesian cells that are not transition cells standard central differences are used for slope reconstruction. 
No limiting is applied.

We note that on cut cells, one reconstructs to the midpoint of the cut cell edges, not to the midpoint of the Cartesian edges.
To be more precise, the update on a cut cell is given by
\begin{align}\label{eq: trap on cut cell 2d}
\begin{split}
S_{ij}^{n+1} &= S_{ij}^n - \frac{\Delta t}{\alpha_{ij}\Delta x \Delta y}
\biggl[ F_{i+1/2,j}^{n+1/2,T} - F_{i-1/2,j}^{n+1/2,T} + G_{i,j+1/2}^{n+1/2,T} - G_{i,j-1/2}^{n+1/2,T}  \biggr]
\end{split}
\end{align}
where  $\alpha_{ij} \in (0,1)$ is the volume fraction and
\begin{multline*}
F_{i+1/2,j}^{n+1/2,T} = \frac{1}{2}u\beta_{i+1/2,j}\Delta y (S_{i+1/2,j}^{n} + S_{i+1/2,j}^{n+1}),\hspace*{4.0cm}\\
S_{i+1/2,j}^{n} = \hspace*{10.3cm} \\ 
\begin{cases}
S_{ij}^{n} + (x_{i+1/2,j}-x_{ij}) S_{ij,x}^n + (y_{i+1/2,j}-y_{ij})
S_{ij,y}^n & \text{if} \: u > 0,\\[8pt]
S_{i+1,j}^{n} + (x_{i+1/2,j}-x_{i+1,j}) S_{i+1,j,x}^n + (y_{i+1/2,j}-y_{i+1,j}) S_{i+1,j,y}^n& \text{if} \: u < 0.
\end{cases}
\end{multline*}
Here $(x_{i+1/2,j},y_{i+1/2,j})$ denotes the location of the edge midpoint of face $(i+1/2,j)$,
$(x_{ij},y_{ij})$ denotes the centroid of cut cell $(i,j)$,
and $S_{ij,x}$ and $S_{ij,y}$ refer to the reconstructed unlimited $x$- and $y$-slopes respectively in cell $(i,j)$. Further,
$\beta_{i+1/2,j} \in [0,1]$ represents the area fraction of cut cell edge
$(i+1/2,j)$ compared to a full Cartesian edge. 
The fluxes $G_{i,j+1/2}^{n+1/2,T}$ are defined analogously.
This concludes the brief description of the MUSCL-Trap scheme. More details can be found in \cite{May_Berger_explimpl}.

\begin{remark}[Cost of MUSCL-Trap]\label{remark: cost mixed scheme}
In each time step, one needs to solve an implicit system
that couples the cut cells and transitions cells. Since cut cells occur only at the boundary of the embedded object, the size of this
system is one dimension lower than the overall number of cells. In 2d on a Cartesian grid with $N$ cells in
each direction, one expects $\mathcal{O}(N)$ unknowns in the implicit system. Therefore, using this mixed explicit implicit approach is significantly
cheaper than using an implicit scheme everywhere.
\end{remark}

In \cite{May_Berger_explimpl}, May and Berger presented numerical results for the (unlimited) MUSCL-Trap scheme in 2d.
The $L^1$ error converged with second order. The $L^{\infty}$ error
however showed convergence rates between 1.3 and 1.6. This is {\em not} in line with the second-order accuracy that we saw in 1d.
Newer results with DG codes on cut cell meshes for piecewise \textit{linear} polynomials however have also showed reduced convergence rates in
the $L^{\infty}$ norm of 1.5 to 1.6 \cite{DoD_SIAM_2020,Giuliani_DG}. This raises the general question of whether we can expect to see full second-order convergence in the $L^{\infty}$ norm on 2d cut cell meshes at all but this is the goal.

In the following we want to examine the accuracy of the scheme in 2d more closely than done in \cite{May_Berger_explimpl}.
Here, we will focus on the error caused by switching the time stepping scheme. Due to the complexity of cut cells in 2d,
our examination will be mostly based on numerical experiments.

\begin{remark}
  For our numerical tests in 2d, we use the code setup from \cite{May_Berger_explimpl} as starting point:
  our implementation is based on {\texttt{BoxLib}} \cite{BoxLib}, a library for massively parallel
AMR applications. 
For the generation of the cut cells, we use
\texttt{patchCubes}, a variant of the \texttt{cubes}
mesh generator that is part of the \texttt{Cart3D} package \cite{Cart3d_homepage,aftosmis_j98}. 
The solution of the resulting implicit system is done using \texttt{umfpack} \cite{umfpack}.
\end{remark}

\subsubsection{The transition error for MUSCL-Trap (in absence of cut cells)}

We would like to examine the time stepping transition error in the absence of error sources (2)(a)-(c) caused by the irregularity of the cut cells.
Therefore, we consider the test setup shown in figure \ref{Fig: test 2d transition error}: we create a setup where we switch between explicit and implicit time stepping but all cells are full Cartesian cells.
As on a Cartesian mesh error sources (2)(a)-(c) drop out,
the remaining error should be dominated by the switch from MUSCL to Trapezoidal.

\begin{figure}
 \centering
\includegraphics[scale = 0.57]{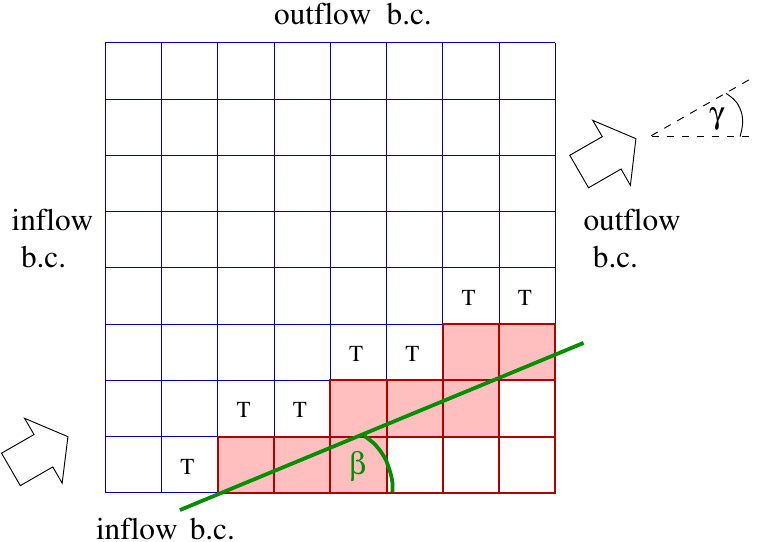}
\caption{Setup for testing transition error in 2d: All cells are Cartesian cells but we simulate the switch in time stepping. To do so, we
  intersect a Cartesian grid with a line with angle $\beta$. All cells cut by this line
  (marked in pink) are treated as `fake' (as Cartesian) cut cells. All `fake' cut cell edges as well as all edges of cells below the `fake' cut cells use the implicit
  scheme (marked in red). All remaining edges (marked in blue), i.e., all edges of cells above the `fake' cut cells use the explicit scheme. The switch takes place on the cells
  marked with a `T'. }
 \label{Fig: test 2d transition error}
\end{figure}

For the theoretical considerations, we focus on the case of 45$^\circ$ and choose $u=v=1$
and $\Delta x = \Delta y$.
The update on a transition cell is given by
\[
S_{ij}^{n+1} = S_{ij}^n - \frac{\Delta t}{\Delta x \Delta y}
\biggl[ F_{i+1/2,j}^{n+1/2,T} - F_{i-1/2,j}^{n+1/2,M} + G_{i,j+1/2}^{n+1/2,M} - G_{i,j-1/2}^{n+1/2,T}  \biggr].
\]A lengthy computation, which involves the usage of the equation $s_t + s_x+ s_y = 0$ to re-express derivatives, leads to the following formula for the
one step error on transition cells:
\[
  L(\bar{s}^{n},\bar{s}^{n+1})_{ij} = \frac{1}{4} ( \dt^2 - \tfrac{\dt^3}{\dx} ) s_{xx}(t^n,x_i,y_j)
  - \frac{1}{4} ( \dt^2 - \tfrac{\dt^3}{\dx} ) s_{yy}(t^n,x_i,y_j) + \bigO(\dt^3).
  \]
  Therefore, we expect in general a one step error of second order. For the special case $s_{xx} = s_{yy}$, we expect a third-order one step error.
  (We confirmed this in numerical experiments not shown here.)

  In numerical tests, see e.g. \textbf{Test \TestCart} below, we often observed (for a variety of angles) that on a full Cartesian mesh, we have a second-order one step error in $L^{\infty}$ due to this transition -- but that this transition error accumulates `only' to a scheme of order $\mathcal{O}(h^{1.5})$ for the $L^{\infty}$ error at time $T$.

\subsection{New mixed explicit implicit schemes}

We now present the extensions of MUSCLmod-Trap and MPRKC-Trap to two dimensions.

\subsubsection{The MUSCLmod-Trap scheme}

To derive the new explicit MUSCLmod scheme, we essentially do a truncation error analysis of Trapezoidal rule and MUSCL scheme on a mesh as shown in figure \ref{Fig: test 2d transition error} and
design MUSCLmod to contain the second-order transition error terms. One needs to be a bit careful in this computation as the original MUSCL scheme is based on an unsplit version, involving corner coupling and a more evolved evaluation of transverse derivatives. This then results in the following change of the reconstructed value $S_{i+1/2,j}^{M,n+1/2}$ (the approximation to the solution at the midpoint of edge $(i+1/2,j)$ at time $t^{n+1/2}$ used in the MUSCL scheme) for the case $u,v>0$
\[
  S_{i+1/2,j}^{M\text{mod},n+1/2} =S_{i+1/2,j}^{M,n+1/2} + \frac{\Delta t^2}{4} \left( u^2 S_{ij,xx}^n + 2uv S_{ij,xy}^n \right) 
  - \frac{\Delta t \Delta x}{4} \left( u S_{ij,xx}^n + v S_{ij,xy}^n \right) .
\]
The second derivatives are computed using standard second-order difference quotients on Cartesian cells away from cut cells. On transition cells we fit a quadratic polynomial to compute them.

Numerical tests show for MUSCLmod as fully explicit scheme on a Cartesian mesh a third-order one step error and standard second-order accuracy at time $T$. The numerical stability limit also seems to be very similar to the
original unsplit MUSCL.

\subsubsection{The MPRKC-Trap scheme}
We now extend the explicit MPRKC scheme, given by \eqref{eq: MPRKC}-\eqref{eq: flux explicit} in 1d, to 2d. We use the unsplit
MUSCL scheme described by Almgren et al. \cite{almgrenBellSzymczak:1996}, which we also use for the MUSCL-Trap scheme, to compute the
predictor $S^{(1)}$. Afterwards, we use the two-dimensional analogue of \eqref{eq: flux explicit} for the computation of the
fluxes $F_{i\pm 1/2,j}^{n+1/2,ET}$ and $G_{i,j \pm 1/2}^{n+1/2,ET}$.

The classic two-stage second-order SSP RK scheme in combination with central differences for $\Delta x = \dy = h$ is 
stable with $\nu \le 1.0$ under the CFL condition \cite{Berger_Helzel_2012}
\begin{equation}\label{eq: compute dt MPRKC}
\dt=\nu \frac{h}{\abs{u}+\abs{v}}
\end{equation}
due to using a split approach. 
The new MPRKC scheme uses the unsplit MUSCL as predictor (instead of the split upwind scheme) but the split explicit Trapezoidal scheme as corrector.
We therefore expect MPRKC to
be at least as stable as the classic two-stage second-order SSP RK scheme and to be in particular stable under the
CFL condition \eqref{eq: compute dt MPRKC}. We confirmed this in numerical tests.

Next, we briefly sketch the extension of the mixed MPRKC-Trap to 2d.
If we compare the 1d sketches of MUSCL-Trap, see figure \ref{fig: switch scheme: flux bounding},
and MPRKC-Trap,
see figure \ref{Fig: MPRKC-Trap}, we observe that the implicit zone has been extended by two cells. This holds also true in 2d:
Given $S_{ij}^n$, 
\begin{enumerate}[label=(\roman*)]
\item compute explicit fluxes using the MUSCL scheme and update all cells that have been treated fully explicitly
  by MUSCL-Trap to ${S}_{ij}^{(1)}$;
\item use the predicted values ${S}_{ij}^{(1)}$ for taking the second step of the explicit MPRKC scheme;
  due to slope reconstruction, this will `cost' two layers of cells; therefore, the position of the transition cells has been
  shifted by two cell layers to the interior of the flow domain, compare figure \ref{Fig: switching scheme in 2d};
\item compute all implicit fluxes for the extended implicit zone; update cut cells, fully implicitly treated implicit cells
  in the extended implicit zone, and transition cells to $S_{ij}^{n+1}$.
\end{enumerate}
Note that the MPRKC-Trap scheme has been constructed to have a third-order one step error on Cartesian meshes.

\subsection{Numerical results in 2d}

We will consider two different tests in the following: We will first compare the various mixed schemes on a mesh that contains only Cartesian cells.
This is to confirm our analytical considerations above. Afterwards we will compare the mixed schemes on a mesh that contains cut cells.

\subsubsection{Test on Cartesian mesh}

\textbf{Test \TestCart:} We consider the setup shown in figure \ref{Fig: test 2d transition error}. We set $\beta=30^{\circ}$. The Cartesian mesh is chosen to cover $[0,1]^2$ with $\Delta x = \Delta y$ and $N$ denotes the number of Cartesian cells in one coordinate direction, i.e., $N = \frac{1}{\Delta x}$. The ramp starts at approximately $x_0=0.146$.
We use as test function $s_0 = 1 + \exp(-120*((x-0.49)^2+(y-0.20)^2))$ so that the peak of the Gaussian aligns with where we switch the scheme. We choose the velocity field parallel to the angle $\beta$, i.e., $\gamma=\beta$, and set
$(u,v)^T=(2,2\tan(\beta))$. We use $T=0.15$. The initial data and the solution at the final time is shown in figure
\ref{Fig: test function Test Cart}.
We use $\nu=0.8$. The time step $\Delta t$ for the mixed schemes involving MUSCL and MUSCLmod is computed using \eqref{eq: compute dt MUSCL}, for the mixed scheme involving MPRKC we use \eqref{eq: compute dt MPRKC}. 
We compare MUSCL-Trap, MUSCLmod-Trap, and MPRKC-Trap. We also include results using the explicit versions of the scheme everywhere (i.e., in that case we do \textit{not} switch to an implicit scheme).

In figure \ref{Fig: Test Cart 1 step error} we show the error after taking one time step, both in the $L^1$ norm and in the $L^{\infty}$ norm. As expected, all fully explicit versions show a third-order one step error. For the mixed schemes, both new versions (MUSCLmod-Trap and MPRKC-Trap) show third order in $L^{\infty}$
whereas MUSCL-Trap only converges with second order.

In figure \ref{Fig: Test Cart time T error} we show the error at time $T$. We observe second-order convergence for the $L^1$ norm for all 6 schemes. The error for MUSCL is smallest, followed by MUSCL-Trap. For the error in the $L^\infty$ norm, all schemes except for MUSCL-Trap show second order. MUSCL-Trap converges on the finest meshes with order 1.5. So there is some error accumulation from the second-order one step error but we only seem to lose half an order. We have observed an order of 1.5 for MUSCL-Trap on Cartesian meshes in several tests. In terms of absolute error sizes, we observe the smallest errors for MUSCL, MUSCLmod, and MUSCLmod-Trap, with them being essentially identical on the finest mesh.

\begin{figure}
  \centering
  \includegraphics[scale = 0.6]{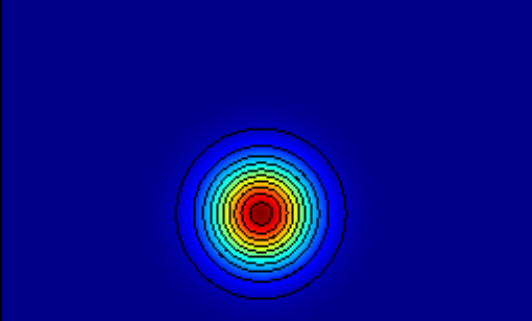}
  \includegraphics[scale = 0.25]{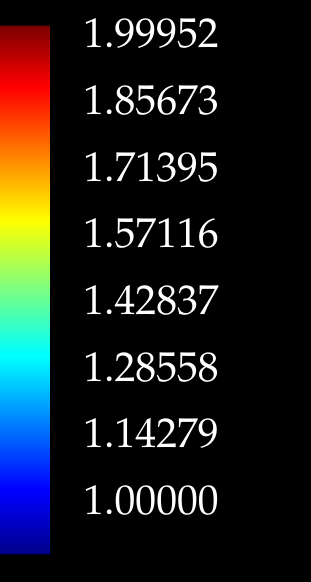}
    \includegraphics[scale = 0.6]{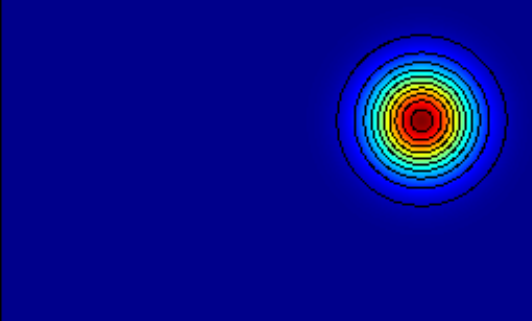}
\caption{Setup \textbf{Test \TestCart}: Initial data (\textit{left}) and solution at $T=0.15$ (\textit{right}) for MUSCLmod-Trap on a 256$\times$256 mesh. (Only the zoom on the domain $[0,1]\times[0,0.6]$ is shown.)}
 \label{Fig: test function Test Cart}
\end{figure}

\begin{figure}
  \centering
  \includegraphics[scale = 0.4]{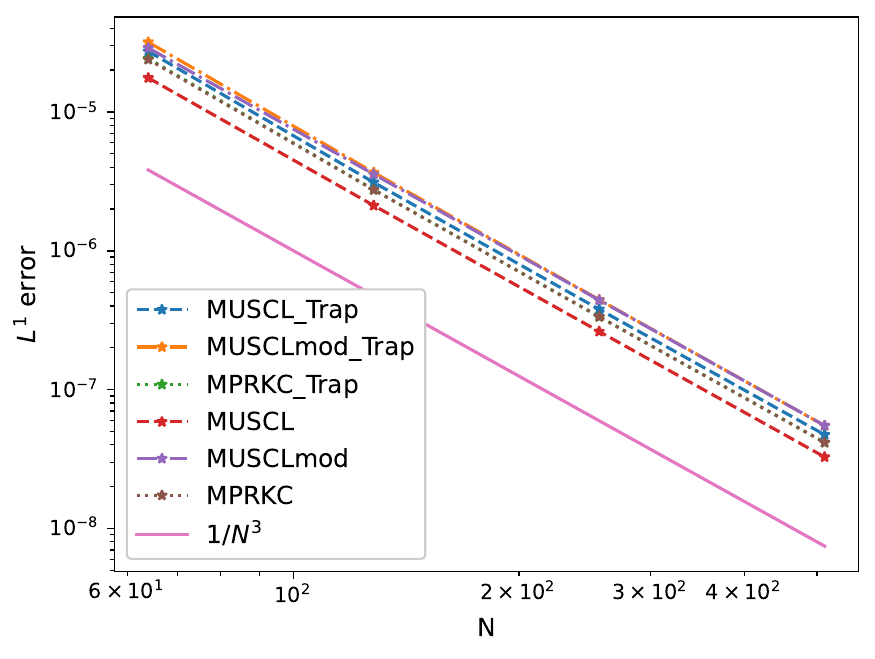}
  \includegraphics[scale = 0.4]{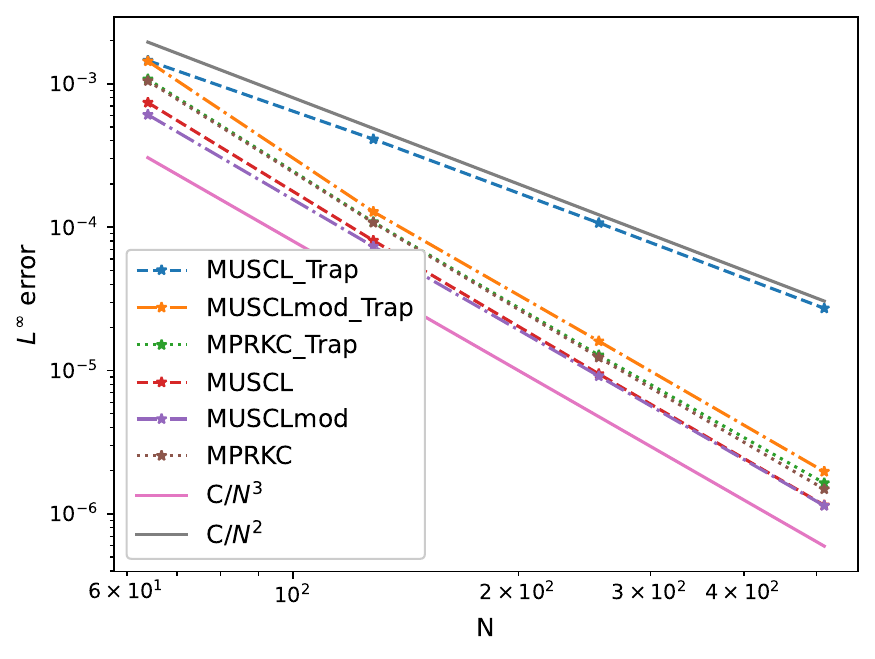}
\caption{Results for \textbf{Test \TestCart}: Error after 1 time step. \textit{Left:} $L^1$ error. \textit{Right:} $L^\infty$ error.}
 \label{Fig: Test Cart 1 step error}
\end{figure}
\begin{figure}
  \centering
  \includegraphics[scale = 0.4]{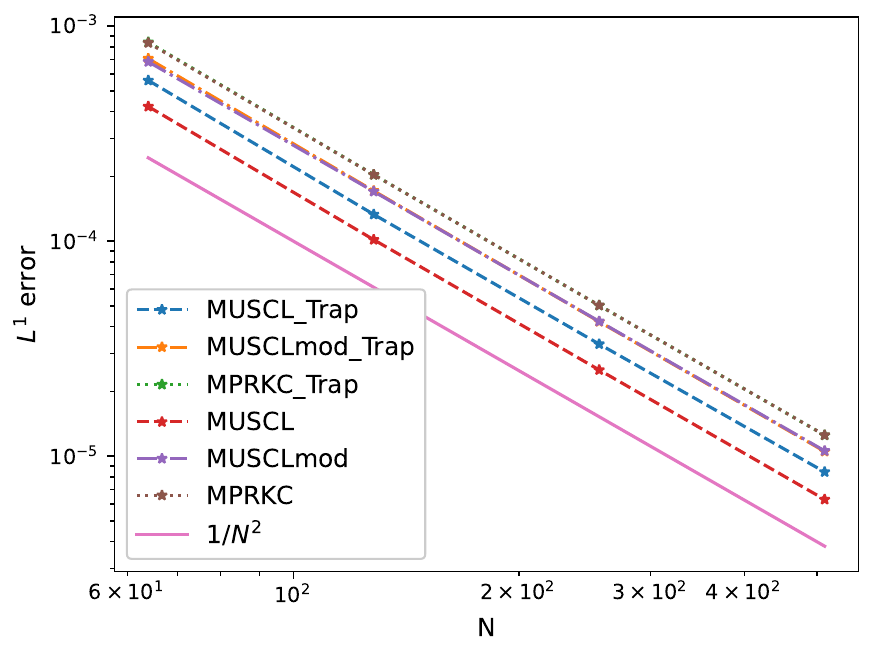}
  \includegraphics[scale = 0.4]{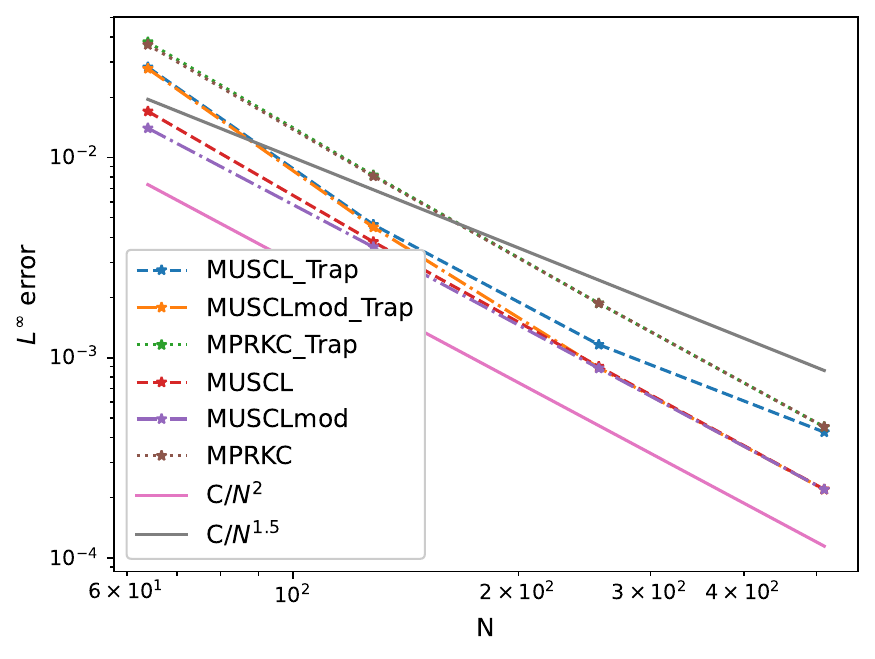}
\caption{Results for \textbf{Test \TestCart}: Error at time $T$. \textit{Left:} $L^1$ error. \textit{Right:} $L^\infty$ error.}
 \label{Fig: Test Cart time T error}
\end{figure}

\FloatBarrier

\subsubsection{Test on cut cell mesh}

Our discussion of the schemes in 2d so far has focused on the transition error between the explicit and implicit
schemes in the absence of cut cells. The goal was to examine and eliminate error source (1) separately from error sources
(2)(a)-(c). We now add cut cells again and 
conclude this contribution with a numerical comparison involving the following three schemes:
MUSCL-Trap, MUSCLmod-Trap, MPRKC-Trap. The goal is to examine whether the reduction of the transition error leads to more accurate results and better convergence orders for the cut cell situation as well.

\textbf{Test \Testcut}: We consider the ramp setup shown in figure \ref{Fig: switching scheme in 2d}. We again choose $\gamma=\beta$. We will consider 4 different angles for this test: 10$^\circ$, 20$^\circ$, 30$^\circ$, and 40$^\circ$. The setup is similar to \textbf{Test \TestCart} with the main difference being that we now have cut cells along the ramp. The Cartesian mesh is chosen to cover $[0,1]^2$ with $\Delta x = \Delta y$. We use as test function again $s_0 = 1 + \exp(-120*((x-x_0)^2+(y-y_0)^2))$. The starting point of the ramp and the centering of the test function  $(x_0,y_0)$ varies for the different angles. They are chosen in such a way that at the initial time, the cut cells are located around the peak, see, e.g., figure \ref{Fig: Test cut time T L1} for the solution at the final time for ramp angle $30^{\circ}$.

Besides the accuracy of the 3 different schemes, we also want to examine the effect of the inaccurate slope reconstruction on cut cells and transition cells. Typically, we use a least squares fit to construct gradients on cut cells and transition cells, see, e.g., \cite{May_Berger_LP}, which is only first-order accurate. As an easy way of obtaining slopes with higher accuracy, we will also include results that use analytic slopes on cut cells and transition cells (together with standard second-order slope reconstruction on Cartesian cells that are updated fully explicitly). These will be marked with `ana' in the legend whereas `LS' implies that the least squares slope reconstruction has been used.

The tests will focus on comparing the accuracy of the different schemes. In terms of cost, for all schemes we need to solve implicit systems of size
$\bigO(N)$. The costs for MUSCL-Trap and MUSCLmod-Trap are pretty comparable. MPRKC-Trap is somewhat more expensive due to applying the explicit scheme twice, the potentially reduced time step length caused by the different CFL conditions, and the extended implicit zone, which is also reflected in somewhat longer running times.

\begin{figure}
  \centering
    \includegraphics[scale = 0.46]{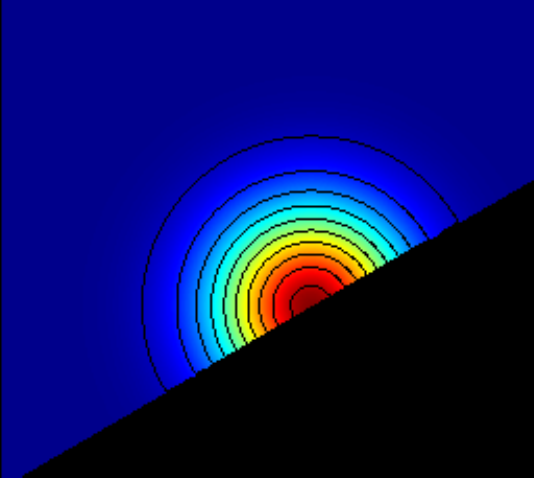}\hspace*{0.2cm}
  \includegraphics[scale = 0.38]{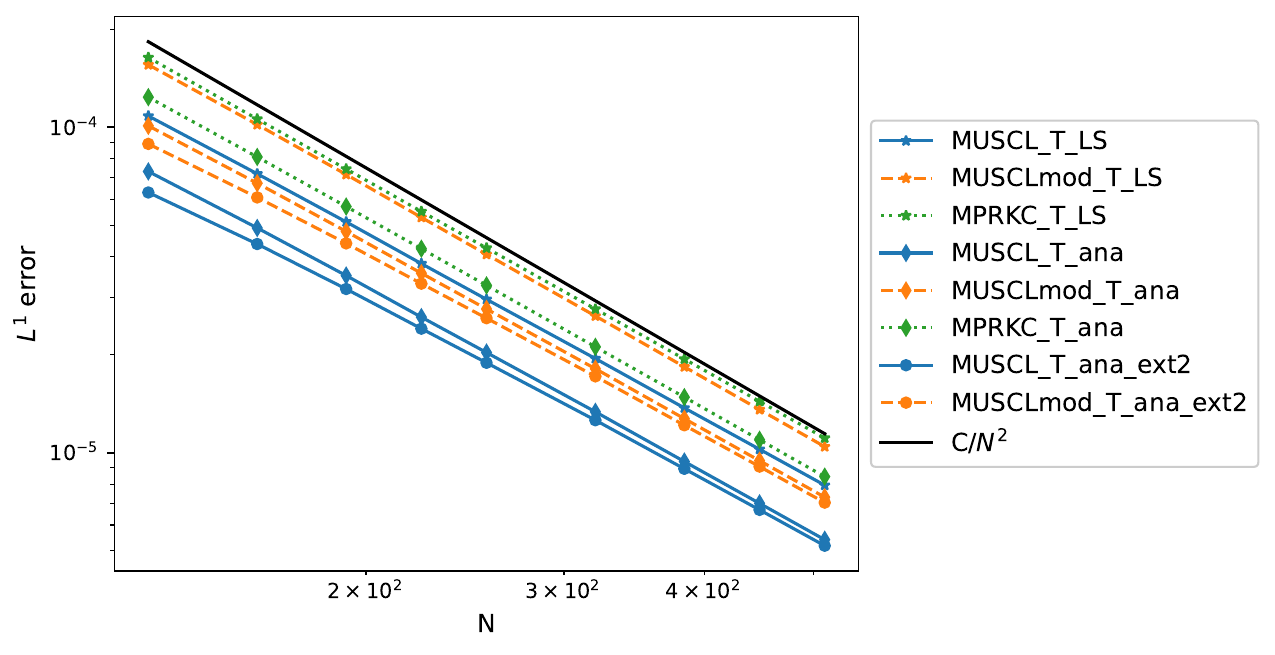}
\caption{Results for \textbf{Test \Testcut} for $30^{\circ}$ ramp: \textit{Left:} Zoom on the solution at time $T$ using MUSCL-Trap for $N=512$. \textit{Right:} $L^1$ error for various schemes.}
 \label{Fig: Test cut time T L1}
\end{figure}

In figure \ref{Fig: Test cut time T L1} we show the results for the error in the $L^1$ norm for the $30^{\circ}$ ramp. The results for the other angles are very comparable. We observe second order convergence for all schemes. Generally, using the original MUSCL scheme as explicit scheme leads to slightly smaller $L^1$ errors for this test.
We have also included results for versions of MUSCL-Trap and MUSCLmod-Trap with an extended implicit zone, where the implicit zone has been extended by 2 cells. These are marked as `ext2'. As a result the set of implicitly treated cells is now the same as for MPRKC-Trap.

  In figure \ref{Fig: Test cut time T Linf} we present the results for the error in the $L^{\infty}$ norm for all 4 angles. Recall that the test function was chosen such that we expect the bigger errors on cut cells, which is actually the case. It is very common for errors on cut cells in the $L^{\infty}$ norm to show a \textit{zig-zag} behavior. In table \ref{Table: 2d test cut LS slopes} we therefore also present slopes that we get from fitting straight lines through the data shown in figure \ref{Fig: Test cut time T Linf} by means of a least squares approach. 
While the details vary a bit for the different ramp angles, we can generally observe that
\begin{itemize}
\item When using the LS slopes, the 3 different schemes (MUSCL-Trap, MUSCLmod-Trap, and MPRKC-Trap) show similar absolute errors and also similar convergence orders (with the exception of the $40^{\circ}$ ramp). Having improved the transition error does \textit{not} result in smaller errors overall for this test.
\item Using analytic slopes reduces the error (in terms of its size) significantly, with factors varying between 4 and 10.
\item For MPRKC-Trap, using analytic slopes instead of LS slopes improves the convergence orders, within a range of 0.15 to 0.42.
\item For MUSCL-Trap and MUSCLmod-Trap, we do not see significant improvement in convergence orders when using analytic slopes instead of LS slopes. For extending the implicit zone by 2 cell layers, we do see significant improvement for the angles $10^{\circ}$ and $20^{\circ}$ degrees but not for the higher ones.
  \item We observe the highest convergence orders for MPRKC-Trap with analytic slopes. In terms of errors sizes it is very comparable though with MUSCL-Trap and MUSCLmod-Trap when using analytic slopes and an extended implicit zone.
  \end{itemize}
  To sum it up: despite the improved transition error, we do \textit{not} see a significantly  improved convergence order or significantly smaller errors for the new versions MUSCLmod-Trap and MPRKC-Trap. Using analytic slopes leads to significantly smaller error sizes but not necessarily significantly improved convergence orders for this test.

\begin{figure}
  \centering
  \includegraphics[scale = 0.39]{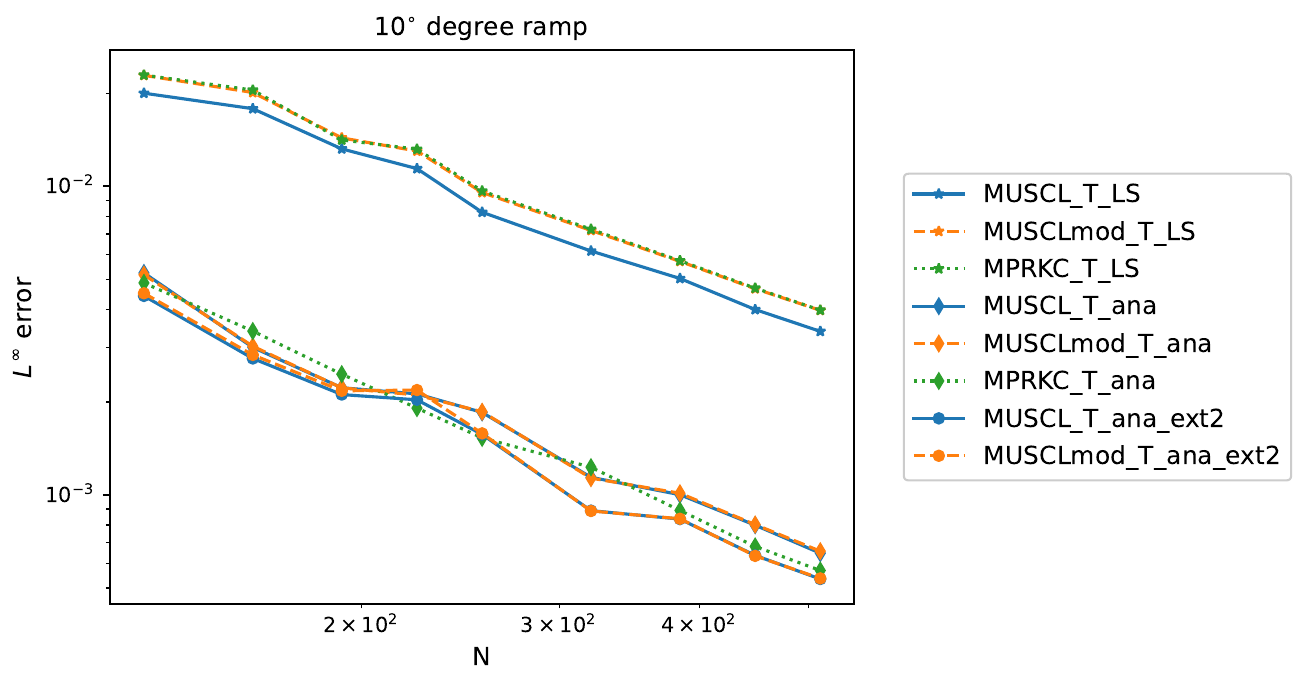}\hspace*{-0.15cm}
  \includegraphics[scale = 0.39]{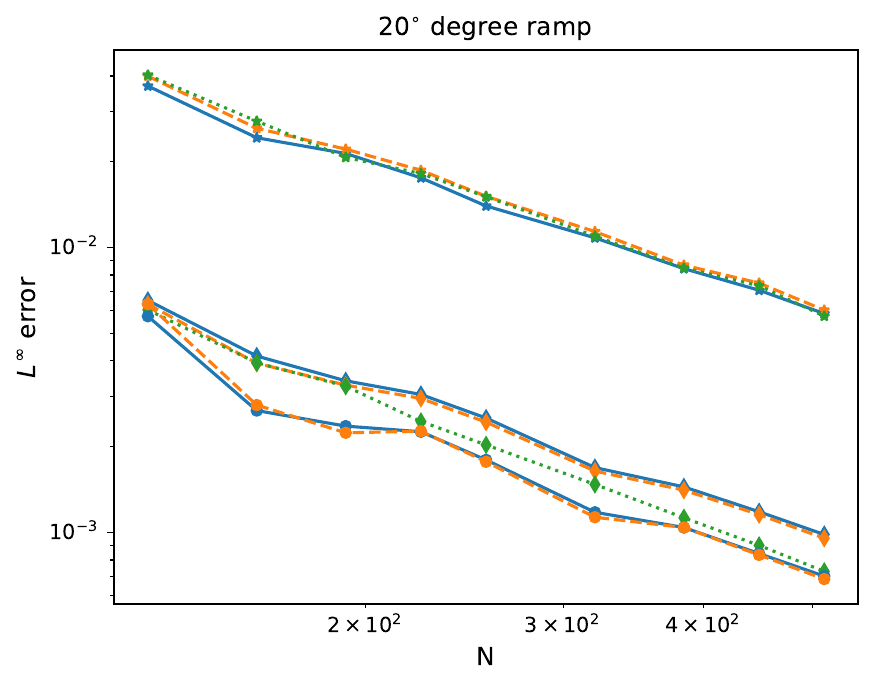}
  \includegraphics[scale = 0.39]{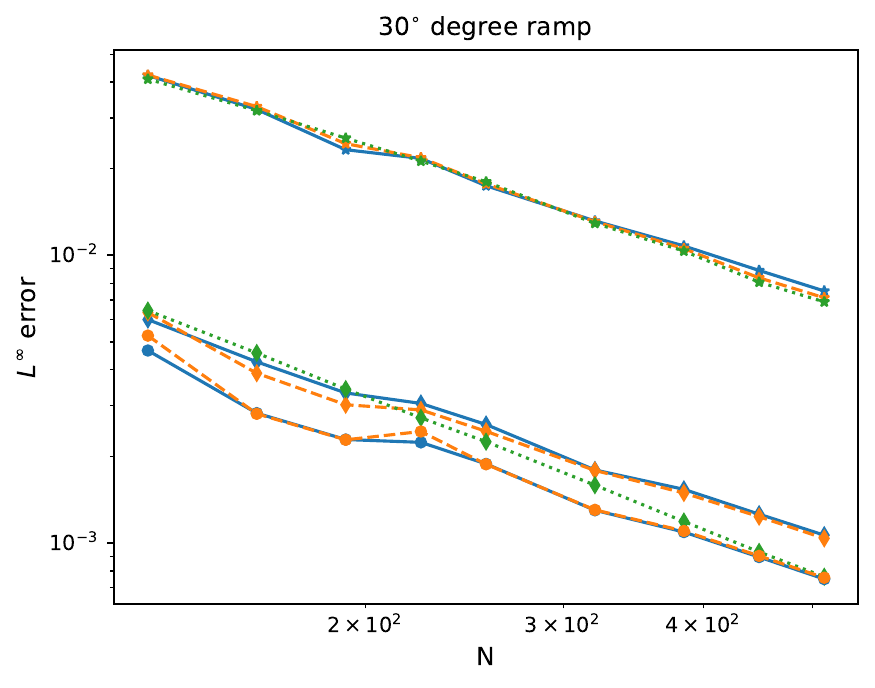}
    \includegraphics[scale = 0.39]{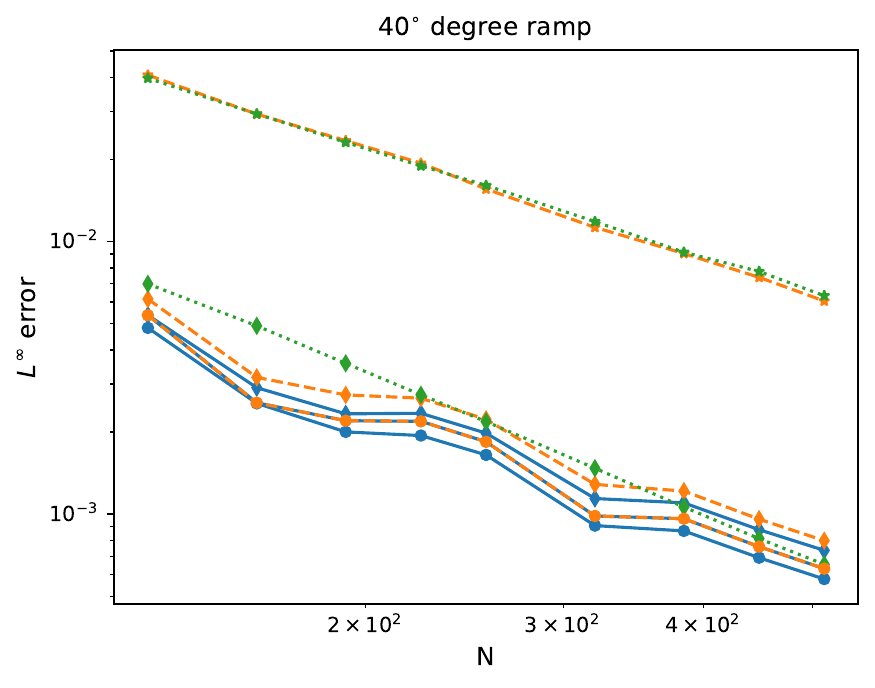}
    \caption{Results for \textbf{Test \Testcut}: Error in $L^{\infty}$ norm. Convergence orders are given in table \ref{Table: 2d test cut LS slopes}.}
 \label{Fig: Test cut time T Linf}
\end{figure}

\begin{table}[b]
  \begin{center}
    {\small
    \caption{Result of \textbf{Test \Testcut}: Convergence orders for errors in $L^{\infty}$ norm, computed as a least squares fit for data shown in figure \ref{Fig: Test cut time T Linf}}
\label{Table: 2d test cut LS slopes}
\begin{tabular}{crcccccc}
\hline 
Method & $10^{\circ}$ & $20^{\circ}$ & $30^{\circ}$ & $40^{\circ}$ \\
\hline
\\[-0.2cm]
MUSCL\_Trap\_LS          & 1.36  & 1.29  & 1.23  & 1.42\\
MUSCLmod\_Trap\_LS  & 1.33  & 1.32 & 1.29 & 1.37\\
MPRKC\_Trap\_LS       & 1.33   & 1.35 & 1.30 & 1.31 \\
MUSCL\_Trap\_ana     & 1.41  & 1.32 & 1.22 & 1.34 \\
MUSCLmod\_Trap\_ana & 1.39 & 1.31 & 1.22 & 1.37 \\
MPRKC\_Trap\_ana    & 1.53  & 1.50 & 1.54 & 1.73 \\
MUSCL\_Trap\_ana2   & 1.50  & 1.37 & 1.24 & 1.44 \\
MUSCLmod\_Trap\_ana2 & 1.53 & 1.43 & 1.28 & 1.42 \\
\hline
\end{tabular}  }   
\end{center}
\end{table}

\FloatBarrier

\section{Conclusions and Future Plans}\label{sec: outlook}

In this contribution, we analyzed the accuracy of the mixed explicit implicit scheme consisting of MUSCL as explicit scheme and Trapezoidal rule with slope reconstruction as implicit scheme. For the one step error, we identified several second-order error sources linked to the irregular size of the cut cells as well as a second-order transition error when switching from explicit to implicit schemes. For linear advection in 1d, we can show that these errors do not accumulate in the usual way and that the resulting scheme is second-order accurate.

This is not the case in 2d. We therefore introduced two new mixed schemes, MUSCLmod-Trap and MPRKC-Trap, with improved transition errors that are based on exchanging the explicit scheme. When using the mixed scheme on a fully Cartesian mesh, this led to improved convergence orders in 2d. When using these new schemes on a test involving cut cells however there was no significant difference to using the original MUSCL-Trap. Using analytic slopes led to a slight improvement of convergence orders and a significant improvement in the actual error size. Note though that the improved convergence orders are not that different from the 
newer results with DG codes on cut cell meshes for piecewise \textit{linear} polynomials, which show reduced convergence orders in
the $L^{\infty}$ norm of 1.5 to 1.6 \cite{DoD_SIAM_2020,Giuliani_DG} as well. So maybe it is to much to aim for full second order in $L^{\infty}$ for the ramp test. 

More intensive numerical tests are needed to make a definite statement. It currently seems that it might only pay off to reduce the transition error if also the other errors sources (caused by the irregularity of the cut cells) are taken care of. The first step would be to upgrade the slope reconstruction on cut cells and neighbors of cut cells to second-order. This can be achieved by fitting quadratic polynomials. Our initial attempts of implementing this have led to irritating results. It is currently not clear whether there is a bug in the implementation or whether there is some kind of weird interaction going on with the solves necessary for the implicit time stepping. We currently use an approximate Newton scheme for that with the approximate Jacobian being based on a first-order discretization.



\appendix

\subsection*{Conflict of interest}
On behalf of all authors, the corresponding author states that there is no conflict of interest.

\bibliographystyle{plain}
\bibliography{biblio.bib}

\begin{thebibliography}{10}

\bibitem{umfpack}
\url{http://faculty.cse.tamu.edu/davis/suitesparse.html}.

\bibitem{aftosmis_j98}
M.~J. Aftosmis, M.~J. Berger, and J.~E. Melton.
\newblock Robust and efficient {C}artesian mesh generation for component-based
  geometry.
\newblock {\em AIAA Journal}, 36(6):952--960, 1998.

\bibitem{almgrenBellSzymczak:1996}
A.~S. Almgren, J.~B. Bell, and W.~G. Szymczak.
\newblock A numerical method for the incompressible {N}avier-{S}tokes equations
  based on an approximate projection.
\newblock {\em SIAM J. Sci. Comput.}, 17(2):358--369, March 1996.

\bibitem{Barth_LS_3d}
T.J. Barth.
\newblock A 3-d least-squares upwind {E}uler solver for unstructured meshes.
\newblock In M.~Napolitano and F.~Sabetta, editors, {\em Thirteenth
  International Conference on Numerical Methods in Fluid Dynamics}, volume 414
  of {\em Lecture Notes in Physics}, pages 240--244. Springer, Berlin,
  Heidelberg, New York, 1993.

\bibitem{Berger_Aftosmis_Murman}
M.~Berger, M.~J. Aftosmis, and S.~M. Murman.
\newblock Analysis of slope limiters on irregular grids.
\newblock In {\em 43rd AIAA Aerospace Sciences Meeting, Reno, NV}, 2005.
\newblock Paper AIAA 2005-0490.

\bibitem{Berger_Giuliani_2021}
M.~Berger and A.~Giuliani.
\newblock A state redistribution algorithm for finite volume schemes on cut
  cell meshes.
\newblock {\em J. Comput. Phys.}, 428, 2021.

\bibitem{Berger_Helzel_2012}
M.~J. Berger and C.~Helzel.
\newblock A simplified $h$-box method for embedded boundary grids.
\newblock {\em SIAM J. Sci. Comput.}, 34:A861--A888, 2012.

\bibitem{Berger_Helzel_Leveque_2003}
M.~J. Berger, C.~Helzel, and R.~LeVeque.
\newblock H-{b}ox method for the approximation of hyperbolic conservation laws
  on irregular grids.
\newblock {\em SIAM J. Numer. Anal.}, 41:893--918, 2003.

\bibitem{Burman2010}
E.~Burman.
\newblock Ghost penalty.
\newblock {\em C. R. Math. Acad. Sci. Paris}, 348(21):1217 -- 1220, 2010.

\bibitem{Cart3d_homepage}
\texttt{http://people.nas.nasa.gov/~aftosmis/cart3d/}.

\bibitem{Chern_Colella}
I.-L. Chern and P.~Colella.
\newblock A conservative front tracking method for hyperbolic conservation
  laws.
\newblock Technical report, Lawrence Livermore National Laboratory, Livermore,
  CA, 1987.
\newblock Preprint UCRL-97200.

\bibitem{colella1985}
P.~{Colella}.
\newblock A direct {Eulerian} {MUSCL} scheme for gas dynamics.
\newblock {\em SIAM J. Sci. Stat. Comput.}, 6:104--117, January 1985.

\bibitem{Colella2006}
P.~Colella, D.~T. Graves, B.~J. Keen, and D.~Modiano.
\newblock A {C}artesian grid embedded boundary method for hyperbolic
  conservation laws.
\newblock {\em J. Comput. Phys.}, 211(1):347--366, 2006.

\bibitem{Col_Col_Glaz}
J.P. Collins, P.~Colella, and H.M. Glaz.
\newblock An implicit-explicit {E}ulerian {G}odunov scheme for compressible
  flow.
\newblock {\em J. Comput. Phys.}, 116(2):195--211, 1995.

\bibitem{DoD_SIAM_2020}
C.~Engwer, S.~May, A.~N\"{u}{\ss}ing, and F.~Streitb\"urger.
\newblock A stabilized {D}{G} cut cell method for discretizing the linear
  transport equation.
\newblock {\em SIAM J. Sci. Comput.}, 42(6):A3677--A3703, 2020.

\bibitem{BoxLib}
J.~B.~Bell et~al.
\newblock Box{L}ib {U}ser's {G}uide.
\newblock Technical report, CCSE, Lawrence Berkeley National Laboratory, 2012.
\newblock \texttt{https://ccse.lbl.gov/BoxLib/BoxLibUsersGuide.pdf}.

\bibitem{Frolkovic_2022}
P.~Frolkovi\v{c}, S.~Kri\v{s}kov\'{a}, M.~Rohov\'{a}, and M.~\v{Z}erav\'{y}.
\newblock Semi-implicit methods for advection equations with explicit forms of
  numerical solution.
\newblock {\em Jpn. J. Ind. Appl. Math.}, 39:843--867, 2022.

\bibitem{2d_ghostpenalty}
P.~Fu, T.~Frachon, G.~Kreiss, and S.~Zahedi.
\newblock High order discontinuous cut finite element methods for linear
  hyperbolic conservation laws with an interface.
\newblock {\em J. Sci. Comput.}, 90, 2022.
\newblock article id 84.

\bibitem{Kreiss_Fu}
P.~Fu and G.~Kreiss.
\newblock High order cut discontinuous {G}alerkin methods for hyperbolic
  conservation laws in one space dimension.
\newblock {\em SIAM J. Sci. Comput.}, 43(4):A2404--A2424, 2021.

\bibitem{Giuliani_DG}
A.~Giuliani.
\newblock A two-dimensional stabilized discontinuous {G}alerkin method on
  curvilinear embedded boundary grids.
\newblock {\em J. Sci. Comput.}, 44:A389--A415, 2022.

\bibitem{Klein_cutcell_3d}
N.~Gokhale, N.~Nikiforakis, and R.~Klein.
\newblock A dimensionally split {C}artesian cut cell method for hyperbolic
  conservation laws.
\newblock {\em J. Comput. Phys.}, 364:186--208, 2018.

\bibitem{Gottlieb_Shu}
S.~Gottlieb and C.~Shu.
\newblock Total variation diminishing {R}unge-{K}utta schemes.
\newblock {\em Math. Comput.}, 67:73--85, 1998.

\bibitem{Berger_Helzel_Leveque_2005}
C.~Helzel, M.~J. Berger, and R.~LeVeque.
\newblock A high-resolution rotated grid method for conservation laws with
  embedded geometries.
\newblock {\em SIAM J. Sci. Comput.}, 26:785--809, 2005.

\bibitem{FVCA_Helzel_Kerkmann}
C.~Helzel and D.~Kerkmann.
\newblock An active flux method for cut cell grids.
\newblock In R.~Kl\"ofkorn, E.~Keilegavlen, A.F. Radu, and J.~Fuhrmann,
  editors, {\em Finite {V}olumes for {C}omplex {A}pplications {I}{X} -
  {M}ethods, {T}heoretical {A}spects, {E}xamples}, pages 507--515, Cham,
  Switzerland, 2020. Springer.

\bibitem{Klein_cutcell}
R.~Klein, K.~R. Bates, and N.~Nikiforakis.
\newblock Well-balanced compressible cut-cell simulation of atmospheric flow.
\newblock {\em Philos. Trans. Roy. Soc. A}, 367:4559--4575, 2009.

\bibitem{Krivodonova2013}
L.~Krivodonova and R.~Qin.
\newblock A discontinuous {G}alerkin method for solutions of the {E}uler
  equations on {C}artesian grids with embedded geometries.
\newblock {\em J. Comput. Sci.}, 4(1--2):24--35, 2013.

\bibitem{Laakmann2018}
F.~Laakmann.
\newblock Finite--{V}olumen--{M}ethode zur {L}\"{o}sung von hyperbolischen
  {E}rhaltungsgleichungen auf eingebetteten {G}eometrien.
\newblock Master's thesis, TU Dortmund, 2018.

\bibitem{Leveque02}
R.~J. Le{V}eque.
\newblock {\em Finite Volume Methods for Hyperbolic Problems}.
\newblock Cambridge University Press, Cambridge, UK, 2002.

\bibitem{May_PhD}
S.~May.
\newblock {\em Embedded Boundary Methods for Flow in Complex Geometries}.
\newblock PhD thesis, Courant Institute of Mathematical Sciences, New York
  University, 2013.

\bibitem{FVCA_May}
S.~May.
\newblock Time-dependent conservation laws on cut cell meshes and the small
  cell problem.
\newblock In R.~Kl\"ofkorn, E.~Keilegavlen, A.F. Radu, and J.~Fuhrmann,
  editors, {\em Finite {V}olumes for {C}omplex {A}pplications {I}{X} -
  {M}ethods, {T}heoretical {A}spects, {E}xamples}, pages 39--53, Cham,
  Switzerland, 2020. Springer.

\bibitem{SM_May_Berger_FVCA}
S.~May and M.~Berger.
\newblock A mixed explicit implicit time stepping scheme for {C}artesian
  embedded boundary meshes.
\newblock In J.~Fuhrmann, M.~Ohlberger, and C.~Rohde, editors, {\em Finite
  Volumes for Complex Applications VII-Methods and Theoretical Aspects}, pages
  393--400, Cham, Heidelberg, New York, Dordrecht, London, 2014. Springer.

\bibitem{May_Berger_LP}
S.~May and M.~J. Berger.
\newblock Two-dimensional slope limiters for finite volume schemes on
  non-coordinate-aligned meshes.
\newblock {\em SIAM J. Sci. Comput.}, 35:A2163--A2187, 2013.

\bibitem{May_Berger_explimpl}
S.~May and M.~J. Berger.
\newblock An explicit implicit scheme for cut cells in embedded boundary
  meshes.
\newblock {\em J. Sci. Comput.}, 71:919--943, 2017.

\bibitem{DoD_AMC}
S.~May and F.~Streitb\"urger.
\newblock {D}o{D} stabilization for non-linear hyperbolic conservation laws on
  cut cell meshes in one dimension.
\newblock {\em Appl. Math. Comput.}, 419, 2022.
\newblock 126854.

\bibitem{May_Thein}
S.~May and F.~Thein.
\newblock Explicit implicit domain splitting for two phase flows with phase
  transition.
\newblock {\em Physics of Fluids}, 35:016108, 2023.

\bibitem{Mikula_Ohlberger_Urban}
K.~Mikula, M.~Ohlberger, and J.~Urb\'an.
\newblock Inflow-implicit/outflow-explicit finite volume methods for solving
  advection equations.
\newblock {\em Appl. Numer. Math}, 85:16--37, 2014.

\bibitem{MUSCAT2019108883}
L.~Muscat, G.~Puigt, M.~Montagnac, and P.~Brenner.
\newblock A coupled implicit-explicit time integration method for compressible
  unsteady flows.
\newblock {\em J. Comput. Phys.}, 398:108883, 2019.

\bibitem{Kummer2016}
B.~Müller, S.~Krämer-Eis, F.~Kummer, and M.~Oberlack.
\newblock A high-order discontinuous {G}alerkin method for compressible flows
  with immersed boundaries.
\newblock {\em Internat. J. Numer. Methods Engrg.}, 110(1):3--30, 2016.

\bibitem{Quirk1994}
J.~J. Quirk.
\newblock An alternative to unstructured grids for computing gas dynamic flows
  around arbitrarily complex two-dimensional bodies.
\newblock {\em Comput. \& Fluids}, 23(1):125--142, 1994.

\bibitem{van_Leer_V}
B.~van Leer.
\newblock Towards the ultimate conservative difference scheme, {V}. a second
  order sequel to {G}odunov's methods.
\newblock {\em J. Comput. Phys.}, 32:101--136, 1979.

\bibitem{Wendroff_White}
B.~Wendroff and A.~B. White.
\newblock A supraconvergent scheme for nonlinear hyperbolic systems.
\newblock {\em Computers Math. Applic}, 18(8):761--767, 1989.

\end{thebibliography}
\end{document}